\documentclass[12pt,oneside,reqno]{amsart}
\usepackage[utf8]{inputenc}
\usepackage[english]{babel}

\usepackage[margin=1in]{geometry}

\usepackage{multicol}

\usepackage{amsmath}
\usepackage{amsthm}
\usepackage{amssymb}
\usepackage{graphicx}
\usepackage{mathrsfs}
\usepackage[colorinlistoftodos]{todonotes}
\usepackage[colorlinks=true, allcolors=black]{hyperref}
\usepackage{comment}
\usetikzlibrary{graphs}
\usepackage{float}
\usepackage{xcolor}
\usepackage{tikz-cd}
\usepackage{manfnt}
\usepackage{wasysym}
\usepackage{enumitem}

\makeatletter
\let\@wraptoccontribs\wraptoccontribs
\makeatother

\newcommand*{\squareproof}{Proof}
\newenvironment{myproof}[1][\squareproof]{\begin{proof}[#1]}{\end{proof}}

\newcommand*{\triangleproof}{Proof}
\newenvironment{myprooft}[1][\triangleproof]{\begin{proof}[#1]}{\end{proof}}

\newcommand{\showcomments}{yes}

\newsavebox{\commentbox}
{\ifthenelse{\equal{\showcomments}{yes}}%
{\footnotemark
        \begin{lrbox}{\commentbox}
        \begin{minipage}[t]{1in}\raggedright\sffamily\tiny
        \footnotemark[\arabic{footnote}]}
{\begin{lrbox}{\commentbox}}}%
{\ifthenelse{\equal{\showcomments}{yes}}%
{\end{minipage}\end{lrbox}\marginpar{\usebox{\commentbox}}}
{\end{lrbox}}}

\newtheorem{thm}{Theorem}[section]

\newtheorem{theorem}[thm]{Theorem}
\newtheorem{corollary}[thm]{Corollary}
\newtheorem{lemma}[thm]{Lemma}

\newtheorem{claim}[thm]{Claim}
\newtheorem*{theorem*}{Theorem}

\theoremstyle{definition}

\newtheorem{definition}[thm]{Definition}

\theoremstyle{remark}

\newtheorem{notation}[thm]{Notation}

\newtheorem{remark}[thm]{Remark}

\newcommand{\nclose}[1]{\ensuremath{\langle\!\langle#1\rangle\!\rangle}}
\newcommand{\dist}{\textup{\textsf{d}}}
\newcommand{\scname}[1]{\text{\sf #1}}
\newcommand{\area}{\scname{Area}}

\newcommand{\field}[1]{\mathbb{#1}}
\newcommand{\integers}{\ensuremath{\field{Z}}}

\newcommand{\naturals}{\ensuremath{\field{N}}}
\newcommand{\reals}{\ensuremath{\field{R}}}

\title{The Cohen--Lyndon property in non-metric small-cancellation}

\author{Macarena Arenas}
\contrib[with an appendix by]{Macarena Arenas and Karol Duda}
\address{DPMMS, Centre for Mathematical Sciences, Wilberforce Road, Cambridge, CB3 0WB
 and Clare College, University of Cambridge, Cambridge, CB2 1TL, UK}
\email{mcr59@dpmms.cam.ac.uk}
\address{Department of Mathematics, Universidad del Pa\'is Vasco (UPV/EHU), Spain.}
\email{kduda@impan.pl}

\subjclass[2010]{20F06, 20F67, 20F65}
\keywords{Small-cancellation, Non-positive curvature, Cohen--Lyndon property}
\thanks{The author was supported by a Cambridge Trust \& Newnham College Scholarship, and by the Denman Baynes Research Fellowship at Clare College, Cambridge.}
  
\begin{document}

\begin{abstract}
We show that the Cohen--Lyndon property holds for 
all non-metric small-cancellation quotients. This generalises the analogous result from the metric small-cancellation setting,  and answers a question  asked by  Lyndon in 1966 and by Wall in his 1979 problem list.
\end{abstract}

\maketitle

\section{Introduction}

The goal of this paper is to prove the \emph{Cohen--Lyndon property} for non-metric  small-cancellation quotients of free groups. That is, that given a finite presentation $P=\langle S \mid R \rangle$ satisfying the $C(6)$, $C(4)-T(4)$, or $C(3)-T(6)$ condition, the normal closure $\nclose{R}$  in the free group $F(S)$ on $S$ has a basis consisting of certain conjugates of the elements of $R$. This answers  a question posed by Lyndon~\cite[pg. 222]{Lyn66} and Wall~\cite[Question B5]{Wall79}.

\begin{theorem}\label{thm:CLC6intro}
Let $P=\langle s_1, \ldots, s_n \mid r_1, \ldots, r_k \rangle$ be a $C(6)$, $C(4)-T(4)$, or $C(3)-T(6)$ presentation, and let  $N(\langle r_i\rangle)$ denote the normaliser of $\langle r_i\rangle$ in $\langle  s_1, \ldots, s_n\rangle$.  There exist full left transversals $T_i$ of  $N(\langle r_i\rangle)\nclose{r_1, \ldots, r_k}$  such that $$\nclose{r_1, \ldots, r_k}= \ast_{i \in I, t \in T_i} \langle r_i\rangle ^t.$$
\end{theorem}

The main paper contains the $C(6)$ case of Theorem~\ref{thm:CLC6intro} and the underlying strategy that is used also in the $C(4)-T(4)$ and $C(3)-T(6)$ cases. The appendix, written jointly with K. Duda, contains  the proofs for $C(4)-T(4)$ and $C(3)-T(6)$ small-cancellation presentations.

The analogue to  Theorem~\ref{thm:CLC6intro} in the metric $C'(\frac{1}{6})$ case is due to Cohen and Lyndon~\cite{CL63}. Theorem~\ref{thm:CLC6intro} recovers and generalises Cohen and Lyndon's classical result, and what is more, pushes it beyond the negatively curved setting and into the realm of non-positive curvature. Indeed, finitely presented $C'(\frac{1}{6})$ and $C(7)$ groups are hyperbolic, but while finitely presented $C(6)$,  $C(4)-T(4)$, and $C(3)-T(6)$  groups are non-positively curved by most sensible combinatorial criteria~\cite{GerstenShort90, wilton2022rational, WiseSixtolic}, they are not in general  negatively curved in any meaningful way - indeed, these groups can contain $\integers^2$ subgroups.
 
 For each $n \in \naturals$, the  \emph{metric} and  \emph{non-metric} small-cancellation conditions, denoted by  $C'(\frac{1}{n})$ and  $C(n)$ respectively, measure the overlaps (repeated subwords) between relators in a group presentation; we formulate these conditions precisely in Definition~\ref{def:sccs}. The smallest numbers for which the $C'(\frac{1}{n})$ and  $C(n)$ conditions produce a useful theory are $n=6$, and $n=7$, but there are significant differences between the properties that they are known to satisfy in each case, and between the methods that can be used to approach them. 
 
The $C'(\frac{1}{n})$ condition implies the $C(n+1)$ condition, but in general, the converse is not true: the ``purely non-metric'' $C(n)$ condition does not imply the  $C'(\frac{1}{n'})$ condition for any  choices of $n \geq 2$ and $n' \geq 2$. For instance, for $K \geq 3$, consider the  presentation
 $$P=\langle a,b,t\mid w_K:=ta^{K}t^{-1}b^{K+4} \rangle.$$  This presentation satisfies the $C(6)$ condition, but does not even satisfy  the $C'(\frac{1}{2})$  condition,   because $b^{K+3}$ is a piece and $|w_K|=2(K+3)$. Similar examples of $C(n)$--not--$C'(\frac{1}{2})$ presentations can  easily be produced for any $n\geq 6$.  

Another striking difference between $C(n)$ groups and their metric $C'(\frac{1}{6})$ cousins is that the $C'(\frac{1}{6})$  condition implies cocompact cubulability, while the question of whether  there is any $n$ for which all $C(n)$ groups are cubulated is still open in general~\cite{WiseSmallCanCube04}. 

The $C'(\frac{1}{n})$ and $C(n)$ conditions can also be coupled with the \emph{$T(n)$ condition}, described in Definition~\ref{def:T}, to get a useful theory. Particularly, the $C(4)-T(4)$ and $C(3)-T(6)$ conditions provide similar structural results to the $C(6)$ case. Examples of groups admitting $C(4)-T(4)$ small-cancellation presentations include  prime alternating link groups~\cite{Weinbaum71} and 2-dimensional right angled Artin groups. Encompassing these two classes of examples, non-positively curved square complexes satisfy the $C(4)-T(4)$ condition. In the $C(3)-T(6)$ case, a rich family of examples was constructed by Gersten and Short in~\cite{GerstenShort90}. These examples have Property (T), and are thus very different from the $C(4)-T(4)$ examples.

The Cohen--Lyndon property was introduced in~\cite{CL63}, where it was shown to hold  for one-relator presentations and $C'(\frac{1}{6})$  small-cancellation presentations. It was later established for certain presentations of Fuchsian groups by Zieschang in~\cite{Zieschang66}.  
Informally, the Cohen-Lyndon property encodes when the relators in the presentation are ``as independent as possible''. When none of the relators are proper powers, it implies that the relation module $\nclose{\mathcal{R}}/\nclose{\mathcal{R}}'$ is a sum of $|R|$  cyclic modules~\cite[10.3]{LS77}   and therefore immediately implies that such presentations have no relation gap, and that the corresponding presentation complexes are aspherical. The Cohen-Lyndon property was also used by Baumslag~\cite{Baumslag67} to obtain a partial characterisation of  Hopfian one-relator groups, and by Zieschang to study automorphisms of Fuchsian groups.

We derive the Cohen--Lyndon property from a  homotopical statement, in that we essentially produce am explicit homotopy equivalence  between the Cayley graph associated to the small-cancellation presentation and a wedge of cycles representing the relators and their translates, ranging over  left transversals as in the statement of Theorem~\ref{thm:CLC6intro}. The main technical statement is thus  Lemma~\ref{lem:c6cont}. 
This text therefore serves two purposes: it extends the results in~\cite{CL63} to the setting of non-metric small-cancellation and it introduces a topological viewpoint to proving that the Cohen--Lyndon property holds -- a viewpoint which, we hope, can be used in other settings.

\begin{remark}
While in the present paper we only concern ourselves with group presentations, we note that a version of the Cohen--Lyndon property can be defined for any pair $(G,\mathcal{H})$, or equivalently, for any quotient $G/\nclose{\mathcal{H}}$, where $G$ is an arbitrary group and $\mathcal{H}$ is a collection of subgroups of $G$.    It has been shown to hold in various settings (see~\cite{KarrassSolitarCL72, MCS86, EH87, DuncanHowie94, Sun20}), but none of these variations can be used to deduce Theorem~\ref{thm:CLC6intro}.

The results in the present work extend almost verbatim to the more general setting of graphical $C(6)$ quotients (see for instance~\cite{Gruber15}) of free groups, by replacing the bouquet $B_n$ associated to a generating set $S$ in a ``classical" group presentation with the defining graph in the graphical presentation. To avoid excessive technicalities, we have chosen not to structure our exposition around that version of the theory. In~\cite{Arenas2023pi}, we prove versions of these results for $C(9)$ cubical small-cancellation quotients (in the sense of~\cite{WiseIsraelHierarchy}); we note that the graphical and classical versions of the theory are  special cases of the cubical version.
\end{remark}

\subsection{Structure and strategy:} In Section~\ref{sec:back} we present the necessary background regarding the Cohen--Lyndon property  and small-cancellation theory. In Section~\ref{sec:main} we describe an ordering on the cycles of the Cayley graph that takes into account the structure of a certain simplicial complex associated to the presentation. A key property of this ordering is that it respects (graph-theoretical) distance with respects to a fixed ``origin''; this is proven in Lemma~\ref{lem:well-defined}. The main technical result is Lemma~\ref{clm:contractibleinduction} which allows us  to inductively ``rebuild'' the Cayley graph in a manner consistent with the ordering. These results, together with a couple more lemmas, are then assembled to deduce  Theorem~\ref{thm:CLC6intro}. In the Appendix (Section~\ref{appendix}), we explain how to adapt the results in the previous sections to the $C(4)-T(4)$ and $C(3)-T(6)$ cases.

\subsection*{Acknowledgements} I am grateful to  MurphyKate Montee, Nansen Petrosyan, and the anonymous referee for their careful reading of the proofs contained herein, which  led to many improvements and clarifications. I am also grateful to Henry Wilton and Andrei Jaikin-Zapirain for their suggestions, 
to Daniel Wise for pointing out a reference, and to Sami Douba for providing stylistic guidance.

\section{Background}\label{sec:back}

\subsection{Small-cancellation notions}\label{subsec:c6cl}

We adopt a topological, rather than combinatorial, viewpoint for defining classical small-cancellation theory. This is mostly a matter of convenience -- the topological viewpoint is more suitable for the method or our proof, and generalises more naturally to other versions of the theory. We emphasise that in the classical setting, both viewpoints are equivalent.

Let $P=\langle S \mid R \rangle$ be a presentation for a group $G$, and let $\mathcal{X}(P)$ denote its presentation complex. This is a $2$-complex that has a single vertex, an edge for each $s \in S$, and a $2$-cell for each $r \in R$, so that $\pi_1\mathcal{X}(P)=G$. 
The Cayley graph $Cay(G,S)$ is the $1$-skeleton of the universal cover $\widetilde{\mathcal{X}(P)}:=\widetilde{\mathcal{X}}(P)$. The definitions below are stated in terms of arbitrary $2$-complexes, but the reader may take $X=\widetilde{\mathcal{X}}(P)$ for the remainder of this section.

A map $f: X \longrightarrow Y$ between 2-complexes is \emph{combinatorial} if it maps open cells homeomorphically to open cells.
A complex is \emph{combinatorial} if all attaching maps are combinatorial (possibly after subdividing).

\begin{definition}[Pieces]
Let $X$ be a combinatorial $2$-complex.
A non-trivial combinatorial path $p \rightarrow X$ is a \emph{piece} if there are $2$-cells $C_1, C_2$ such
that $p\rightarrow X$ factors as $p \rightarrow \partial C_1 \rightarrow X$ and $p \rightarrow \partial C_2 \rightarrow X$ but there does not exist
a homeomorphism $\partial C_1 \rightarrow \partial C_2$ such that the following diagram commutes
\[\begin{tikzcd}
p \arrow[d] \arrow[r]  & \partial C_1 \arrow[d] \arrow[ld] \\
\partial C_2 \arrow[r] & X                                
\end{tikzcd}\]
\end{definition}

\begin{definition}[Disc diagram]\label{def:dd} A \emph{disc diagram} $D$ is a compact contractible combinatorial 2-complex, together with an embedding $D \hookrightarrow S^2$ that induces a cell structure on $S^2$. Viewing the sphere as the 1-point compactification of $\reals^2$, this cellular structure consists of the 2-cells of $D$ together with an additional 2-cell containing the point at infinity. The \emph{boundary path} $\partial D$ is the attaching map of the 2-cell at infinity.  A \emph{disc diagram in a complex X} is a combinatorial map $D \rightarrow X$. The \emph{area} of a disc diagram $D$ is the number of $2$-cells in $D$.  
\end{definition}

\begin{remark}
A \emph{spherical diagram} $S$ is defined similarly to a disc diagram, except that $S$ is assumed  to be simply-connected rather than contractible. Note that all the definitions concerning disc diagrams, such as the notion of reducibility defined below, also make sense for spherical diagrams.
\end{remark}

A disc diagram $D$ might map to $X$ in an `ineffective' way: it might, for instance, be quite far from an immersion. Sometimes it is possible to replace a given diagram with a simpler diagram $D'$ having the same boundary path as $D$, as we now explain.

\begin{definition}\label{def:cancellablecell}
A \emph{cancellable pair} in $D$ is a pair of $2$-cells $C_1, C_2$ meeting along a path $e$ such that the following diagram commutes:
 
 \[\begin{tikzcd}
                         & e \arrow[ld] \arrow[rd] &                          \\
\partial C_1 \arrow[rd] \arrow[rr] &                         & \partial C_2 \arrow[ld] \\
                         & X                       &                         
\end{tikzcd}\]

 A cancellable pair leads to a new disc diagram by removing $e\cup Int(C_1)\cup Int(C_2)$ and then glueing together the paths $\partial C_1-e$ and $\partial C_2-e$. This procedure results in a diagram $D'$ with $\area(D')=\area(D)-2$ and $\partial D'=\partial D$.
A diagram is \emph{reduced} if it has no cancellable pairs.
\end{definition}

\begin{definition}[Annular diagrams and collared diagrams]
An \emph{annular diagram} $A$ is a compact combinatorial 2-complex homotopy equivalent to $S^1$, together with an embedding $A \hookrightarrow S^2$, which induces a cellular structure on $S^2$. The \emph{boundary paths} $\partial_{in}A$ and $\partial_{out}A$ of $A$ are the attaching maps of the two 2-cells in this cellulation of $S^2$ that do not correspond to cells of $A$.
An \emph{annular diagram in a complex X} is a combinatorial map $A \rightarrow X$.
A disc diagram $D\rightarrow X$ is \emph{collared} by an annular diagram $A \rightarrow X$ if $\partial D=\partial_{in} A$. 
\end{definition}

An \emph{arc} in a diagram is a path
whose internal vertices have valence $2$ and whose initial and terminal vertices have valence $\geq 2$. 
A \emph{boundary arc} is an arc that lies entirely in $\partial D$.

\begin{definition}[Small-cancellation conditions]\label{def:sccs}
A complex $X$ satisfies the $C(n)$ condition if for every reduced disc diagram $D \rightarrow X$, the boundary path of each $2$-cell
in $D$ either contains a non-trivial boundary arc, or is not the concatenation of less than
 $n$ pieces.
The complex $X$ satisfies the $C'(\frac{1}{n})$ condition if for each $2$-cell $R \rightarrow X$, and
each piece $p \rightarrow X$ which factors as $p \rightarrow R \rightarrow X$, then $|p| < \frac{1}{n}|\partial R|$.
\end{definition}

When a complex $X$ satisfies sufficiently good small-cancellation conditions, then it is possible to classify the reduced disc diagrams $D \rightarrow X$ in terms of a few simple behaviours exhibited by their cells.

\begin{figure}
\centerline{\includegraphics[scale=0.53]{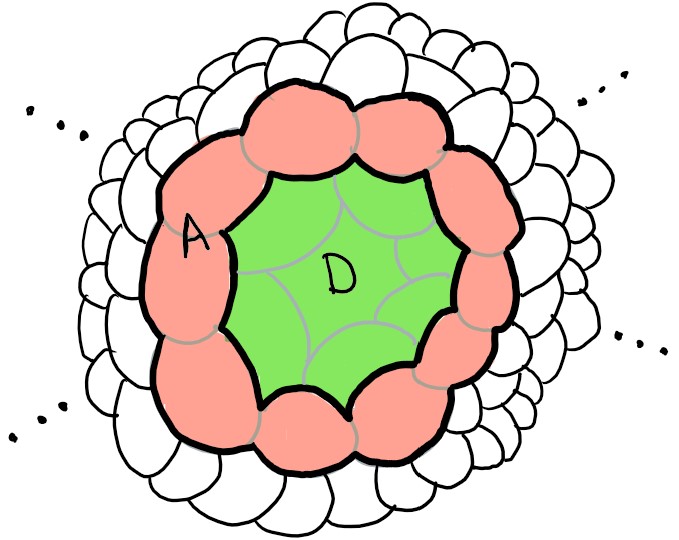}}
\caption{An annular diagram collaring a disc diagram in a $C(6)$ complex.}
\label{fig:annulardisc}
\end{figure}

\begin{definition}[Ladders, shells, and spurs]
A disc diagram $L$ is a \emph{ladder} if it is the union
of a sequence of closed $1$-cells and $2$-cells $C_1,\ldots,C_n$, such that for $1<j<n$,
there are exactly two components in $L - C_j$ , and exactly one component in
$L - C_1$ and $L - C_n$. Moreover, if $C_i$  is a $1$-cell  then it is not contained in any
other $C_j$. 

A \emph{shell} of $D$ is a 2-cell $C \rightarrow D$
whose boundary path $\partial C \rightarrow D$ is a concatenation $qp_1 \cdots p_k$ for some $k \leq 3$
where $q$ is a boundary arc in $D$ and $p_1, \ldots , p_k$ are non-trivial pieces in the interior of $D$.
The arc $q$ is the \emph{outerpath} of $C$ and the concatenation $p_1 \cdots p_k$ is the \emph{innerpath} of $C$. 
A \emph{spur} is a vertex of degree $1$ on $\partial D$.
\end{definition}

We now state the two fundamental results of small-cancellation theory. Proofs can be found in~\cite{Lyn66,McCammondWiseFanLadder}, for instance.

\begin{theorem}[Greendlinger's Lemma]\label{thm:greendlingerc6}
Let $X$ be a $C(6)$ complex and $D \rightarrow X$ be a reduced disc diagram, then either
\begin{enumerate}
\item $D$ is a single cell,
\item $D$ is a ladder,
\item $D$ has at least three shells and/or spurs.
\end{enumerate}
\end{theorem}

\begin{theorem}[The Ladder Theorem]\label{thm:ladderc6}
Let $X$ be a $C(6)$ complex and $D \rightarrow X$ be a reduced disc diagram. If $D$ has exactly $2$ shells or spurs, then $D$ is a ladder.
\end{theorem}

We say a diagram $D$ in $X$ has \emph{minimal area} if it has minimal area amongst all diagrams $D' \rightarrow X$ with $\partial D=\partial D'$. Note that a minimal area disc diagram must be reduced, but that the converse is not necessarily true.

\section{Main Theorem}\label{sec:main}

We start with the following remark, which we state and prove in a way that is tailored exactly to our applications, and which is known to experts in other similar settings (see for instance~\cite[5.6+5.7]{WiseIsraelHierarchy}):

\begin{lemma}\label{lem:c6cont}
Let $X$ be a simply-connected $C(6)$ small-cancellation complex. Let $C_1,C_2$ be 2-cells of $X$. Then  either $\partial C_1 =\partial C_2$, or $ C_1\cap  C_2=\emptyset$, or $ C_1\cap  C_2$ is contractible.
\end{lemma}

\begin{proof}
Assume that $C_1 \cap C_2\neq \emptyset$. To show that $C_1 \cap C_2$ is connected, we proceed by contradiction.

For vertices $v,v' \in X^{(1)}$,  let $\dist(v,v')$ denote the usual graph metric, i.e.,  $\dist(v,v')$ is the least number of edges in a path connecting $v$ and $v'$ in $X$ (note that such a path exists because $X$ is simply connected). Let $a,b $ be vertices of $C_1 \cap C_2$ in distinct connected components satisfying that $\dist(a,b)$ is minimised amongst any such pairs of points.   Let $\alpha \rightarrow C_1, \beta \rightarrow C_2$ be paths with endpoints $a,b$. Furthermore, choose $\alpha, \beta$ to be locally geodesic (so that they have no spurs) and so that the disc diagram $D$ bounded by $\alpha\beta^{-1}$ has minimal area amongst all possible such choices. Note that if  $D$  contains either $C_1$ or $C_2$, then either $\alpha=\beta^{-1}$, in which case $a,b$ lie in the same connected component of  $C_1 \cap C_2$, contradicting the initial hypothesis, or we can find a smaller area disc diagram satisfying the conditions above and not containing the corresponding $C_i$ by "pushing across" to expel $C_i$ from $D$. Consider the disc diagram $D_+$ obtained by attaching $C_1$ and $C_2$ to $D$ along $\alpha$ and $\beta$. Then $D_+$ is a ladder by Theorem~\ref{thm:ladderc6}. Indeed, if $D_+$ were not reduced, then since $C_1,C_2$, and $D$ are reduced, a cancellable pair would contain one of $C_1,C_2$ and a 2-cell in $D$, but then, after performing the cancellation, we would obtain a new diagram $D'$ with the same properties as $D$ and smaller area, contradicting our previous choice.  Now, since $D$ was assumed to be minimal area, then Greendlinger's Lemma implies that $D$ is either a single cell, a ladder, or has at least three shells or spurs.  

If $D$ has a shell $C$ this contradicts the $C(6)$ condition, as its outerpath is then the concatenation of at most $2$ pieces (since neither $C_1$ and $C_2$ are contained in $D$, then the pieces are components of the intersection of $C$ with $C_1$ and $C_2$). Thus, $D$ is either a single vertex, in which case $a=b$, or $D$ has exactly two spurs (in the ladder case) or at least three spurs (in the general case). If $D$ has $\geq 3$ spurs, then at least one of these lies either on $\alpha$ or $\beta$, contradicting that these paths are locally geodesic. Thus $D$ has exactly $2$ spurs. In this case, by removing the spurs from the diagram we find vertices $a',b'$ in the same connected components of $C_1 \cap C_2$ as $a,b$ but such that  $\dist(a',b')< \dist(a,b)$, contradicting the  hypothesis that $a,b$ had been chosen to minimise the distance between points in the respective connected components of $C_1 \cap C_2$. 

Since in all cases we arrive at a contradiction, then we conclude that $C_1 \cap C_2$ has a single connected component, which is either a single piece (and hence simply connected and contractible)  or, up to removing backtracks, $\partial C_1 = \partial C_2$.
\end{proof}

We return to the setting of group presentations. A  presentation $P=\langle S \mid R\rangle$  is \emph{symmetrised} if whenever $r \in R$, then $\bar r \in R$ whenever $\bar r$ is a cyclic permutation of $r$ or $\bar r=r^{-1}$; the presentation $P$ is \emph{cyclically reduced} if no $r \in R$ contains a subword $w=ss^{-1}$ where $s \in S$.

A finite presentation  $P$ is a $C(n)$ \emph{small-cancellation presentation} if it is symmetrised and cyclically reduced, and $\mathcal{X}(P)$ satisfies the $C(n)$ condition.

Mainly to establish the notation that will be used later on, we review the construction of the Cayley complex associated to a group presentation.
Let $P=\langle s_1, \ldots, s_n \mid R=\{r_1, \ldots, r_k \}\rangle$ and let $B_n$ denote a bouquet of $n$ loops labelled by the generators $s_1, \ldots, s_n$. Recall that the \emph{Cayley complex} $\widetilde{\mathcal{X}}(P)$ for $P$  is the universal cover of the complex obtained by coning-off the cycles $c_1 \rightarrow B_n, \ldots, c_k\rightarrow B_n$ corresponding to the relations $r_1, \ldots, r_k$ in $B_n$. The subgroup  $ker(F_n \rightarrow G(P))= \langle \langle r_1, \ldots, r_k  \rangle \rangle$ is associated to a regular covering space $\hat B_n \rightarrow B_n$, and $\widetilde{\mathcal{X}}(P)$ is obtained from $\hat B_n$ by coning-off the set $\{g\tilde{c_i}\}_{i \in I, g\nclose{R} \in F_n/\nclose{R}}$.  In other words, $\hat B_n$ is the Cayley graph of $G(P)$.

Before stating our main result, we make a standard observation, which we prove for the sake of completeness.

\begin{lemma}\label{lem:c6embedding}
Let $P=\langle s_1, \ldots, s_n \mid r_1, \ldots, r_k \rangle$ be a $C(6)$ small-cancellation presentation. Then each lift $g\tilde{c_i} \rightarrow \widetilde{\mathcal{X}}(P)$  of $c_i \rightarrow \mathcal{X}(P)$ embeds in $\widetilde{\mathcal{X}}(P)$.
\end{lemma}

\begin{proof}
Assume $g\tilde{c_i} \rightarrow \widetilde{\mathcal{X}}(P)$ is not an embedding, let $\sigma$ be a non-closed subpath of $c_i$ whose induced lift $\tilde \sigma  \rightarrow \widetilde{\mathcal{X}}(P)$ factors through $g\tilde{c_i}\rightarrow \widetilde{\mathcal{X}}(P)$ and bounds a disc diagram $D$, which we can assume to be of minimal area. Let $\tilde \beta  \rightarrow \widetilde{\mathcal{X}}(P)$ be such that $\tilde \sigma \tilde \beta=g\tilde{c_i}$.  Let $D'$ be a minimal area disc diagram with boundary path $\tilde \sigma \tilde \beta$ and finally let $D''=D \cup D'$. Note that $D''$ is  reduced, as otherwise we would be able to cancel $C$ with a 2-cell of $D$, contradicting that $D$ had minimal area. Since  the endpoints of $\sigma$ are assumed to be distinct, Theorem~\ref{thm:greendlingerc6} implies that $D''$ is either a ladder or has at least $3$  spurs. Both cases lead to a contradiction, since the only shell on $\partial D''$ is $C$ and $g\tilde{c_i}$ cannot contain any spurs, as $P$ is cyclically reduced.
\end{proof}

\begin{notation}
In view of Lemma~\ref{lem:c6embedding}, we may drop the  ``\texttildelow'' from the notation and write $g c_i $ to denote a translate $g\tilde{c_i}$ of a cycle $\widetilde c_i$ in $\hat B_n$.
\end{notation}

The rest of this section is dedicated to proving that the Cohen--Lyndon property holds for $C(6)$ quotients of free groups. This is Theorem~\ref{thm:CLC6intro} from the introduction, which we restate below. 

Recall that a subset $T \subset G$ is a \emph{full left transversal for $H$ in $G$} if and only if every left coset of $H$ contains exactly one element of $T$. Let $S=\{s_1, \ldots, s_n\}$ and let $r_i$ be a word on $S \cup S^{-1}$. We use the notation $N(\langle r_i\rangle)$ to denote the normaliser of $\langle r_i\rangle$ in the free group $F(S)$.

\begin{theorem}\label{thm:CLC6}
Let $P=\langle s_1, \ldots, s_n \mid  r_1, \ldots, r_k \rangle$ be a $C(6)$ small-cancellation presentation. Then $(\langle  s_1, \ldots, s_n\rangle,   \{ \langle r_i \rangle\}_i  )$ has the Cohen--Lyndon property. 
That is, there exist full left transversals $T_i$ of  $N(\langle r_i\rangle)\nclose{r_1, \ldots, r_k}$ in $F(S)$ such that 
$$\nclose {r_1, \ldots, r_k}= \ast_{i \in I, t \in T_i} \langle r_i\rangle ^t.$$
\end{theorem}

We now introduce the objects and constructions derived from the presentation $P=\langle s_1, \ldots, s_n | r_1, \ldots, r_k \rangle$ that will be used in the proof of this theorem.

\begin{definition}[The structure graph]\label{def:structure}
Define the \emph{structure graph} $\Lambda$ of the presentation complex $\mathcal{X}(P)$ as follows. The vertex set $V(\Lambda):=V$ of $\Lambda$ has two types of vertices: \begin{enumerate}
\item The first type corresponds to translates $gc_i$ of the cycles $c_1, \ldots, c_k$ in $\hat B_n$. The set of vertices of this type will be denoted $V_T$. 
\item To describe the second type, we start by defining the \emph{untethered hull} of $\hat B_n$. This is the subcomplex $\mathcal{F}$ of $\hat B_n$  consisting of  all 1-cells of $\hat B_n$ that do not lie in any $gc_i$. In general, $\mathcal{F}$ is not connected. An \emph{untethered component} is a connected component $\mathcal{F}_\iota$  of $\mathcal{F}$. The second type of vertices of $\Lambda$ corresponds to the untethered components of $\mathcal{F}$. The set of vertices of this type will be denoted $V_U$.
\end{enumerate}  

Let $X_v$ denote the subcomplex of $\hat B_n$ corresponding to $v \in V$. Let $$\mathcal{U}=\{X_v : v \in V\}.$$ 
Note that $\mathcal{U}$ is a topological cover of $\hat B_n$.

The edges of $\Lambda$ are also of two types: they either correspond to non-empty intersections  $gc_i \cap g'c_j$, or $gc_i \cap \mathcal{F}_\iota$ where  $g, g'$ range over the left cosets of the $\langle r_i \rangle$'s,  $i, j \in I$, and $\iota $ ranges over the connected components of $\mathcal{F}$. 
By Lemma~\ref{lem:c6cont}, each intersection $gc_i \cap g'c_j$ is either a piece $p$ or a vertex, so an edge of  $\Lambda$ corresponds to a piece in $\hat B_n$, or to a vertex that arises from either an intersection between a pair of distinct $gc_i$'s, or from an intersection between a $gc_i$ and an untethered component of $\mathcal{F}$ (indeed, a non-empty intersection $gc_i \cap  \mathcal{F}_\iota$ must be a single vertex because $\hat B_n$ is the Cayley graph of $G(P)$). Note also that $\Lambda$ is a simplical graph. Indeed,  $\Lambda$ has no bigons by construction, and no loops by Lemma~\ref{lem:c6embedding}.
\end{definition}

\begin{remark}\label{rmk: contractible components} Note that, by construction, each $ \mathcal{F}_\iota$ is a tree, and hence contractible. 
Lemma~\ref{lem:c6cont} shows that each intersection $gc_i \cap g'c_j$ is also contractible.
\end{remark}

There is a Helly property for the elements of $\mathcal{U}$. It can be proven using induction and Greendlinger's Lemma, and can also be found in~\cite[6.11]{OP18}:

\begin{lemma}\label{clm:helly}
Let $\mathcal{X}(P)$ be the presentation complex associated to a $C(6)$ small-cancellation presentation $P$, let $\Lambda$ be its structure graph, and let $V' \subset V(\Lambda)$ be finite. If each pairwise intersection $ X_v \cap X_v'$ with $v,v' \in V'$ is non-empty, then the total intersection $\bigcap_{v \in V'}X_v$ is non-empty. 
\end{lemma}

The following lemma, a version  of which is proven in~\cite[6.12]{OP18}, provides finer information about the total intersections in Lemma~\ref{clm:helly}. We include a proof since it is illustrative of the methods of this paper.

\begin{lemma}\label{clm:maxints}
Let $\mathcal{X}(P)$ be the presentation complex associated to a $C(6)$ small-cancellation presentation $P$ and let $\Lambda$ be its structure graph. For all $I \subset V(\Lambda)$ with $\infty >|I|\geq 2$, the intersection $\bigcap_{v \in I} X_v $ is non-empty if and only if it is connected (and hence contractible). 
\end{lemma}

\begin{proof}
We proceed by induction on $|I|$. The base case is $|I|=2$ and follows from Lemma~\ref{lem:c6cont} and from the fact that a non-empty intersection $gc_i \cap  \mathcal{F}_\iota$ must be a single vertex because $\hat B_n$ is the Cayley graph of the presentation.
Assume that the claim holds for intersections $\bigcap_{v \in I} X_v \neq \emptyset$ where $|I|\leq k$, and consider $\bigcap_{v \in I\cup\{v_{k+1}\}} X_v$. 
If this intersection is empty, then there is nothing to show. So we may assume that $\bigcap_{v \in I\cup\{v_{k+1}\}} X_v\neq \emptyset$, and thus contains a vertex $b$ of $\hat B_n$. 
The induction hypothesis implies that each intersection $\bigcap_{v \in J} X_{v_j}$  with $J \subsetneq I$ and $I=J \cup \{v_j\}$ is contractible, and is in fact a tree as each $X_v$ is $1$-dimensional. 
As $\bigcap_{v \in I\cup\{v_{k+1}\}} X_v \subset \bigcap_{v \in I} X_v$, it suffices to show that $\bigcap_{v \in I\cup\{v_{k+1}\}} X_v$ is connected. 
To this end, let $x,y$ be vertices in $\bigcap_{v \in I\cup\{v_{k+1}\}} X_v$. 

Consider paths $\beta \rightarrow X_{v_{k+1}}$ and $\alpha \rightarrow \bigcap_{v \in I} X_v$ with endpoints $x,y$, so that $\beta\alpha^{-1}$ bounds a disc diagram $D$ in $\widetilde{\mathcal{X}}(P)$.  Moreover, amongst all pairs of paths with endpoints $x,y$ as above, choose $\beta$ and $\alpha$ to have minimal length. 
Consider the disc diagram $D^+=D \cup_\beta X_{v_{k+1}} \cup_\alpha X_v$, where $v \in I$. We claim that $D^+$ is a degenerate disc diagram, that is, it contains no $2$-cells, so $\alpha=\beta$ and $x$ and $y$ lie in the same connected component of $\bigcup_{v \in I\cup\{v_{k+1}\}} X_v$. 
To prove this assertion, note that the choices of $\alpha$ and $\beta$ above imply that $\partial D$ cannot contain any spurs, as such spurs could be removed to shorten $\alpha$ and/or $\beta$. 
Thus $D$ must contain a shell $S$, and the outerpath of $S$ is either a subpath of $\alpha$, a subpath of $\beta$, or contains either $x$ or $y$. In the first and second cases, $\partial S \cap \partial D$ is a single piece, and the innerpath of $S$ is the concatenation of at most $3$ pieces, so $\partial S$ is the concatenation of at most $4$ pieces, contradicting the $C(6)$ condition. In the third case, similarly, $\partial S \cap \partial D$ is the concatenation of at most $2$ pieces, so the innerpath of $S$ is the concatenation of at most $3$ pieces, and $\partial S$ is the concatenation of at most $5$ pieces, again contradicting the $C(6)$ condition.

Thus, as asserted, $D$ must be a degenerate disc diagram, so $\bigcap_{v \in I\cup\{v_{k+1}\}} X_v$ is connected and contractible and the induction is complete.
\end{proof}

 As noted in Definition~\ref{def:structure}, the set $\mathcal{U}=\{X_v : v \in V\}$ provides a topological cover of $\hat B_n$. The structure graph $\Lambda$, being simplicial by construction and by Lemma~\ref{lem:c6embedding}, is the  1-skeleton of the geometric realisation of the nerve complex of  $\mathcal{U}$.

\begin{definition}[The nerve of $\mathcal{U}$]
The \emph{nerve complex} $\mathbf{N}(\mathcal{U})$ of a topological covering $\mathcal{U}$ is the abstract simplicial complex $$\mathbf{N}(\mathcal{U})=\{V'  \subset V  : \bigcap_{v \in V'} X_v \neq \emptyset, |V'|< \infty \}.$$
\end{definition}

We note the following  Lemma:

\begin{lemma}\label{lem:hellyimpliesniceunion}
Let $\mathcal{X}(P)$ be the presentation complex associated to a $C(6)$ presentation $P$ and let $\Lambda$ be its structure graph. Let $\sigma$ be a simplex of $\mathbf{N}(\mathcal{U})$ and let $Y_\sigma$  be the subcomplex of $\mathcal{X}(P)$ consisting of all the 2-cells whose boundary corresponds  to vertices in $\sigma$. Then $Y_\sigma$ is simply connected.
\end{lemma}

\begin{remark}
For simplicity, in the proof below we assume $Y_\sigma=\bigcup_{v \in \sigma} C_v$ where each $C_v$  is a 2-cell with $\partial C_v=X_v$ corresponding to a vertex of $\sigma$. In general we could have that more than one 2-cell is glued along $X_v$ (if some relations are proper powers), but this does not affect the simple-connectivity of $Y_\sigma$ (though it could have an effect in its contractibility).
\end{remark}

\begin{proof}
For the first part, the proof is by induction on $|\sigma|$. If $\sigma$ is a single vertex $v$, then $Y_\sigma$ is  a single 2-cell, and thus simply-connected. Similarly when  $|\sigma|=2$, by Lemma~\ref{lem:c6cont}. Assume now that the claim holds whenever $|\sigma|< n$, and let $\sigma$ be an $n$-dimensional simplex,   let $v'$ be a vertex of $\sigma$, and consider the subsimplex $\sigma'$  obtained by removing $v'$  from $\sigma$. 
Then $Y_\sigma=Y_{\sigma'}\cup Y_{v'}$ is simply connected if and only if  $Y_{\sigma'}\cap Y_{v'}=\bigcup_{v \in \sigma'} C_v \cap C_{v'}$ is simply-connected.
But $(\bigcup_{v \in \sigma'} C_v) \cap C_{v'}=\bigcup_{v \in \sigma'}(C_{v}\cap C_{v'})$, so  to prove that this intersection is simply connected, it suffices, by induction, to check that each intersection $(C_{v_i}\cap C_{v'})\cap(C_{v_j}\cap C_{v'})$ is simply connected. But this follows from the Helly property.
\end{proof}

 It is natural to use $\mathbf{N}(\mathcal{U})$ to organise the data of $\mathcal{U}$. Concretely, we do this by defining an  ordering on the vertices of $\mathbf{N}(\mathcal{U})$; this ordering is key to the proof of Theorem~\ref{thm:CLC6}.

\begin{definition}[The ordering on $\mathcal{U}$]\label{def:order}
Choose $v_0 \in V(\Lambda)$ corresponding to a $gc_i$. 
We define a total ordering $\leq$ on $\mathcal{U}$, that is, an injective function $\varphi: V \rightarrow \naturals$. 
To do this, we first define an ordering on the simplices of $\mathbf{N}(\mathcal{U})$ inductively as follows.

Start by setting $\varphi(v_0)=0$, and  define 
$$A_0=\{u \in V\ : \{u,v_0\} \in \mathbf{N}(\mathcal{U})\}\cup\{v_0\}.$$ 
Choose $u_1 \in A_0$, let $\varphi(u_1)=1$, and let 
$$A_{01}=\{u \in V : \{v_0, u_1, u\} \in \mathbf{N}(\mathcal{U})\}\cup\{v_0, u_1\}.$$ Inductively, assume that $\varphi$ has been defined for a subset of cardinality $k$, so $\varphi(v_0)=0, \ldots, \varphi(v_k)=k$. 
For each non-empty simplex $\{v_i, \ldots, v_\ell \}$ of $\mathbf{N}(\mathcal{U})$ where $\varphi(v_i), \ldots, \varphi(v_\ell)$ are already defined, let $$A_{v_i\ldots v_\ell}=\{u \in V : \{v_i, \ldots, v_\ell\}\cup \{u\} \in \mathbf{N}(\mathcal{U})\}\cup \{v_i\ldots v_\ell\}.$$ 
We view each simplex $\{v_i, \ldots, v_\ell \}$ as an ordered tuple $(v_i, \ldots, v_\ell )$ where $\varphi(v_i) < \varphi(v_{i+1}) < \ldots < \varphi(v_\ell)$, and 
order the simplices using the \emph{Lusin–Sierpi\'nski order}\footnote{Also known as the \emph{Kleene–Brouwer order}. Perhaps it would be more suitable to call it the \emph{long-lex order}, in analogy with the \emph{short-lex order}, which is used frequently in geometric group theory.}, which is defined as follows. For a pair of simplices, set $\{v_i, \ldots, v_\ell \}< \{w_{i'}, \ldots, w_{\ell'} \}$ if either 
\begin{enumerate}
\item  there exists $j \leq \min\{\ell,\ell'\}$ with  $v_\iota=w_\iota$ for all $ \iota < j$, and $v_{j}< w_{j}$, or
\item $\ell > \ell'$ and $v_j=w_{j'}$ for all $j'\leq \ell'$.
\end{enumerate}

Now consider a least simplex $\{v_i\ldots v_\ell\}$  such that there exists  $u \in A_{v_i\ldots v_\ell}$ whose image is not yet defined and choose such a vertex $u$ arbitrarily. 
Set $\varphi(u)=k+1$. In Figure~\ref{fig:orderingex} we illustrate the ordering  for a simple example.

A priori, $\varphi$ is only defined for a subset $V' \subset V$; we show in Lemma~\ref{lem:well-defined} that in fact $V' = V$.
\end{definition}

\begin{figure}
\centerline{\includegraphics[scale=0.57]{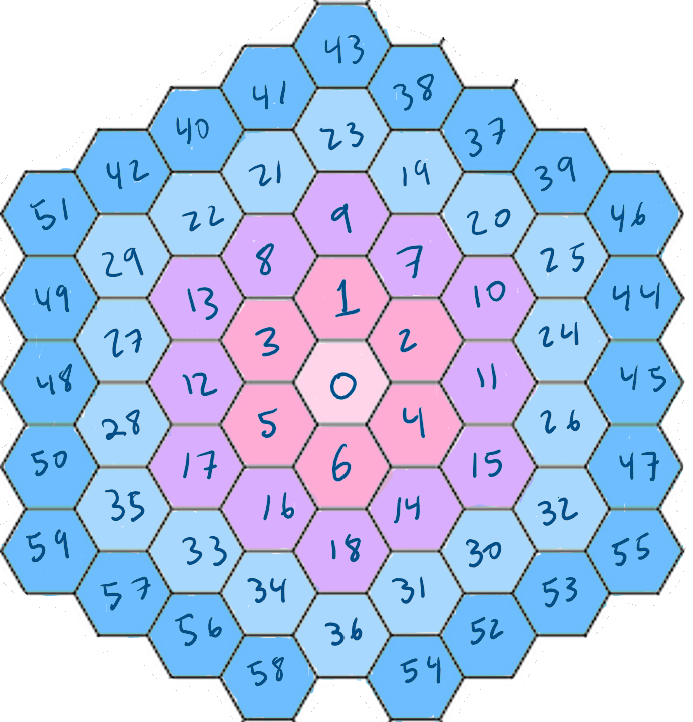}}
\caption{The ordering in Definition~\ref{def:order} for a portion of a hexagonal grid.}
\label{fig:orderingex}
\end{figure}

Recall that, for vertices $v,v' \in V$,  $\dist(v,v')$ denotes the usual graph metric.

\begin{lemma}\label{lem:well-defined} 
Let $v, v' \in V'$. If $\varphi(v)<\varphi(v')$, then $\dist(v, v_0) \leq \dist(v', v_0)$. In particular, the function $\varphi: V \rightarrow \naturals$ is well-defined, i.e.,  $V'=V$.
\end{lemma}

\begin{proof}
We prove the lemma by induction on $N=\varphi(v')$. If  $N=1$ there is nothing to show, so assume that the result holds for all pairs of vertices with image $< N_0$.

Let $v,v'\in V$ and assume $\varphi(v)<\varphi(v')\leq N_0$. 
Note that if $\varphi(v)$ and $\varphi(v')$ are not consecutive integers, then there exists $u \in V$ with $\varphi(v)< \varphi(u)=\varphi(v')-1$, and by the induction hypothesis $\dist(v, v_0) \leq \dist(u, v_0)$, so it would suffice in any case to show that $\dist(u, v_0) \leq \dist(v', v_0)$ to  prove the lemma. Thus, we assume that $\varphi(v)$ and $\varphi(v')$ are  consecutive integers. By definition, there is a simplex $\sigma$ which is least in the Lusin--Sierpi\'nski order described above and such that $\{v'\}\cup \sigma$ is a simplex. 

We first note that for any vertex $u$ adjacent to $v'$ and for any vertex $w'$ in $\sigma$ with $\varphi(w')< N_0$, if $\dist(w', v_0) > \dist(u, v_0)$, then by the induction hypothesis $\varphi(w')>\varphi(u)$. In particular, this holds for the least vertex $w'$ of $\sigma$, so  $\dist(u, v_0) \geq \dist(w, v_0)$ for every vertex $u$ adjacent to $v'$, or otherwise this would contradict that $\sigma$ is the least simplex adjacent to $v'$. Thus, the distance $\dist(v',v_0)$ is realised by a path passing through $\sigma$, and in particular, passing through its least vertex $w'$. There are now several cases to consider:
\begin{enumerate}
\item If $v$ is a vertex of $\sigma$, then $v$ is adjacent to $v'$, and by the discussion on the previous paragraph, $\dist(v, v_0) \leq \dist(v', v_0)$,
\item if $v$ is not a vertex of $\sigma$, then  either
\begin{enumerate}
\item  $\{v\}\cup \sigma$ is also a simplex, in which case, as claimed:
$$\dist(v, v_0)=\dist(v, w')+\dist(w', v_0)=1+\dist(w', v_0)=\dist(v', w')+\dist(w', v_0)=\dist(v',v_0),$$ 
 \item  or  $\varphi(w') < \varphi(v)$. Since $\varphi(v)<\varphi(v')$, the least vertex, say $w$,  adjacent to $v$ must satisfy that $\varphi(w) < \varphi(w')$ and in particular that $\varphi(w)< N_0$. Thus $\dist(w, v_0) \leq \dist(w', v_0)$ by the induction hypothesis applied to  $w'$ and $w$, so
 $$\dist(v,v_0)=1+\dist(w, v_0) \leq 1+\dist(w', v_0)=\dist(v',v_0).$$ 
\end{enumerate}
\end{enumerate}

And the induction is complete. 

This implies that $\varphi$ is defined in all of $V$ since every vertex in the ball $B(r, v_0)$ is smaller in the ordering than an element outside this ball, and we can cover $\mathbf{N}(\mathcal{U})$ by nested balls of increasing radii.
\end{proof}

Write $V=\{v_j\}_{j \in \naturals}$, where the indexing agrees with the ordering in Definition~\ref{def:order}, and recall that  $X_{v_j}$ denotes the subcomplex of $\hat B_n$ corresponding to $v_j \in V$.  The ordering on $V$ induces an ordering on $\mathcal{U}$, 
so $\varphi(v)<\varphi(v')$ if and only if  $X_v<X_{v'}$.
Thus, for ease of notation, in what follows we refer to the ordering of $V$ and the ordering of $\mathcal{U}$ interchangeably.
 
The next lemma is our main technical result. 
 
\begin{lemma}\label{clm:contractibleinduction} Let $\mathcal{X}(P)$ be the presentation complex associated to a $C(6)$ small-cancellation presentation $P$, let $V=\{v_j\}_{j \in \naturals}$ be the vertices of its structure graph $\Lambda$, and for each $v_j \in V$, let $X_{v_j}$ be the corresponding element of $\mathcal{U}$. For each $k \in \naturals$, the intersection
$$\bigcup_{j < k} X_{v_j} \cap X_{v_k}$$ is contractible.
\end{lemma}

\begin{proof}

The proof is by induction on $k$. When  $k=1$ there is nothing to show, and the case of $k=2$ holds because each intersection 
$X_{v} \cap X_{v'}$ is either a piece or a single vertex.

Assume that $\bigcup_{j < k} X_{v_j} \cap X_{v_k}$ is contractible for all $k < k_0$. We now show that $\bigcup_{j < k_0} X_{v_j} \cap X_{v_{k_0}}$ is contractible.

Suppose towards a contradiction that $\bigcup_{j < k_0} X_{v_j} \cap X_{v_{k_0}}$ is not contractible. Then, since this intersection is $1$-dimensional,  it is either disconnected, or there is a non-nullhomotopic loop $\sigma \rightarrow \bigcup_{j < k_0} X_{v_j} \cap X_{v_{k_0}}$. We  note that the second case can be reduced to the first, and is thus precluded by the induction hypothesis.

\textbf{Non-simply-connected intersection:}   
 Suppose  $\sigma \rightarrow \bigcup_{j < k_0} X_{v_j} \cap X_{v_{k_0}}$ is essential (that is, not nullhomotopic), and  let $\bigcup_{j < k_0} C_{v_j}$  denote the subcomplex of $\widetilde \chi(P)$ where $\partial C_{v_j}=X_{v_j}$ for each $j$. 
Then either  $\sigma \rightarrow \bigcup_{j < k_0}  C_{v_j}$ is also essential (as $\sigma \rightarrow  C_{v_{k_0}}$ cannot be essential), or $\bigcup_{j < k_0} C_{v_j}\cup C_{v_{k_0}}$ has the homotopy-type of a 2-sphere. 
The second case is impossible since the Cayley complexes of $C(6)$ small-cancellation groups are aspherical~\cite{Lyn66} (after possibly collapsing collections of 2-cells with the same boundary path) -- indeed, consider a reduced spherical diagram $S \rightarrow \bigcup_{j < k_0} C_{v_j}\cup C_{v_{k_0}}$ that contains at least two 2-cells mapping to distinct 2-cells $C, C'$ of $\chi(P)$ gives rise to a reduced disc diagram $D$ by removing one of the 2-cells, say $C'$. 
But then Greendlinger's Lemma implies that  $D$ contains at least 1 shell (which may be the only 2-cell of $D$), but then the intersection of $C'$ and this shell is a piece, contradiction the $C(6)$ condition. 
Thus, we may assume that $\sigma \rightarrow \bigcup_{j < k_0}  C_{v_j}$ is  essential. The image of  $\sigma$ in $\bigcup_{j < k_0} X_{v_j} \subset \bigcup_{j < k_0} C_{v_j}$ is covered by a collection of arcs in the $X_{v_j}$'s. By the Seifert Van-Kampen Theorem, the image of $\sigma$ cannot be covered by two contractible sets with connected intersection, so in particular, viewing $\sigma$ as a concatenation of two arcs $\tau'$ and $\tau''$ where $\tau'$ traverses a single $X_{v_{j_0}}$ with $1\leq j_0 < k_0$ and $\tau''$ traverses $\bigcup_{j < k_0, j \neq j_0} X_{v_j}$, it follows that the intersection $\bigcup_{j < k_0, j \neq j_0} X_{v_j} \cap X_{v_{j_0}}$ must be disconnected. 
Without loss of generality, we can  choose the $\tau, \tau'$ so that the $X_{V_{J_0}}$ traversed by $\tau'$ is maximal, thus contradicting the induction hypothesis. 

\textbf{Disconnected intersection:} Assume  that there exist vertices $x, y $ lying in distinct connected components of $\bigcup_{j < k_0} X_{v_j} \cap X_{v_{k_0}}$, and let $X_v, X_{v'}$ be the corresponding elements of $\mathcal{U}$ containing $x$ and $y$. 
Note that $x$ and $y$ are connected by a path $\tau$ in $\bigcup_{j < k_0} X_{v_j}$, since by the induction hypothesis, this union is path-connected. 
Let $I_\tau$ be such that $\bigcup_{I_\tau} X_{v_j}$ are the elements of $\mathcal{U}$ traversed by $\tau$, so $\bigcup_{I_\tau} X_{v_j}\cup X_{v_{k_0}}$ is a union of boundaries of $2$-cells in $\widetilde{\mathcal{X}}(P)$. Amongst all possible paths satisfying the above, choose $\tau$ so that $|I_\tau|$ is the least possible, so that $\bigcup_{I_\tau} X_{v_j}\cup X_{v_{k_0}}$  defines an annular diagram $A_\tau$ whose boundary paths are cycles in $\bigcup_{I_\tau} X_{v_j}\cup X_{v_{k_0}}$, and $A_\tau$ collars a reduced disc diagram $D_\tau$ in $\widetilde{\mathcal{X}}(P)$, as in Figure~\ref{fig:annulardisc}.  Choose $D_\tau$ with the fewest cells amongst all disc diagrams collared by $A_\tau$.

If $\area (D_\tau)=0$, then $D_\tau$ is  a tree. Note that if $D_\tau$ has branching, then removing an edge (or several edges corresponding to a piece) from $D_\tau$ corresponds to shortening $\tau$ by pushing it away from a pair $X_{v_j}, X_{v_{j'}}$ and towards $X_{v_{k_0}}$, as in Figure~\ref{fig:cases}; such a reduction contradicts the choices in the previous paragraph, so we may assume that $D_\tau$ is a possibly degenerate subpath of $X_{v_{k_0}}$ with $\partial D_\tau= \tau\tau^{-1}$. Since $\tau\tau^{-1} \subset \bigcup_{j < k_0} X_{v_j} \cap X_{v_{k_0}}$, this contradicts the hypothesis that $x, y $ lie in  distinct connected components of $\bigcup_{j < k_0} X_{v_j} \cap X_{v_{k_0}}$.

Hence, $\area (D_\tau)\geq 1$. We may assume, by performing the same reductions to $\tau$ as in the $\area (D_\tau)= 0$ case, that $\partial D_\tau$ has no spurs, so by Greendlinger's Lemma $D_\tau$ must have at least one shell; we  claim that $\partial S < X_{v_{k_0}}$  for every shell $S$ in $D_\tau$. Assuming the claim (see  Claim~\ref{clm:final}, which is proven below), we now explain how to finish the proof of Lemma~\ref{clm:contractibleinduction}. 
Note that $D_\tau \cup A_\tau$ is reduced by the choice of $\tau$ and $D_\tau$ (as otherwise we would be able to cancel a pair of cells and reduce the area of $D_\tau$).
Let $S$ be a shell  in $D_\tau$. Since $S$ is a shell, it intersects at least $3$ consecutive cells $C_1,C_2,C_3$ of $A_\tau$ as in the centre of  Figure~\ref{fig:induction}, and so, provided $C_2\neq X_{v_{j_0}}$, the path $\tau'$ obtained from $\tau$ by pushing across $C_2$ traverses at most as many cells as $\tau$, but bounds a disc diagram $D_{\tau'}$ with $\area (D_{\tau'})< \area (D_\tau)$, contradicting our initial choices. If $C_2= X_{v_{j_0}}$, then we may choose another shell $S'$ (since we have at least $3$) and run the same argument, and push across the corresponding cell $C_2'$. Since $C_2\neq C_2'$ then $X_{v_{j_0}}\neq C_2'$, so the argument works.

Thus, assuming the claim, we arrive at a contradiction in all cases, so  $\bigcup_{j < k_0} X_{v_j} \cap X_{v_{k_0}}$ must be contractible and the induction is complete. \end{proof}

\begin{figure}
\centerline{\includegraphics[scale=0.65]{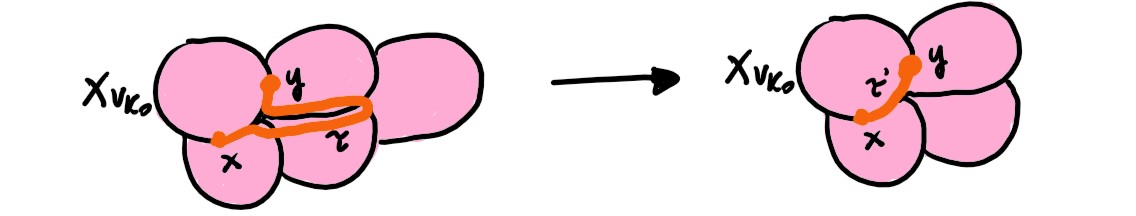}}
\caption{Possible reductions in the last part of the proof of Lemma~\ref{clm:contractibleinduction} when $\area (D_\tau)=0$ and $D_\tau$ has branching. }
\label{fig:cases}
\end{figure} 

\begin{figure}
\centerline{\includegraphics[scale=0.61]{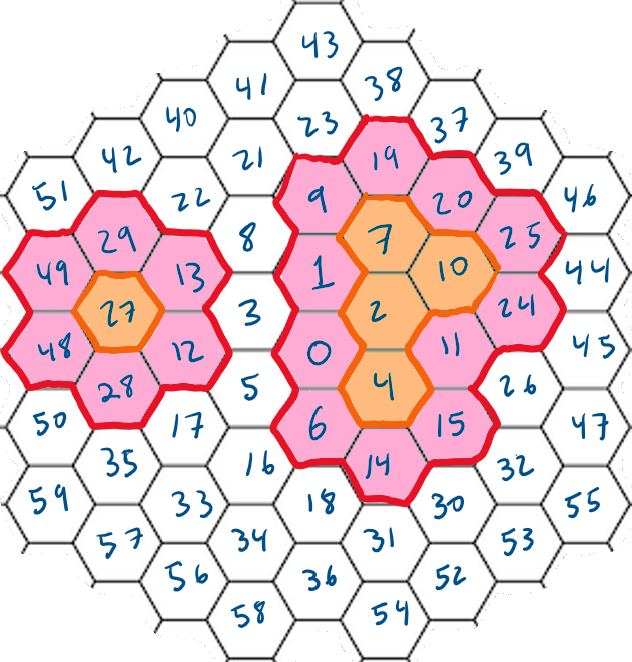}}
\caption{Some annuli collaring disc diagrams in the example from Figure~\ref{fig:orderingex}, and exhibiting the behaviour explained in Claim~\ref{clm:final}.}
\label{fig:orderingannuli}
\end{figure}

\begin{claim}\label{clm:final} Let $\mathcal{X}(P)$ be the presentation complex associated to a $C(6)$ small-cancellation presentation $P$, let $V=\{v_j\}_{j \in \naturals}$ be the vertices of its structure graph $\Lambda$, and for each $v_j \in V$, let $X_{v_j}$ be the corresponding element of $\mathcal{U}$.  Let $A_\beta$ be an annular diagram in $\widetilde{\mathcal{X}(P)}$ collaring a reduced disc diagram $D_\beta$ so that $\partial D_\beta= D_\beta\cap A_\beta=\beta$ and $D_\beta\cup A_\beta$ is reduced, and finally, let $X_{v_{k_0}}$ be the maximal element in the ordering of Definition~\ref{def:order} corresponding to the boundary of a $2$-cell in $A_\beta$. Then $\partial S < X_{v_{k_0}}$  for every shell $S$ in $D_\beta$.
\end{claim}

\begin{proof}[Proof of Claim]
 We prove the claim by induction on $\area (D_\beta)$. 
 For the base case, assume $\area (D_\beta)=1$, so $D_\beta=S$ is a single cell, and in fact, a shell. Consider the union $\bigcup_{j < k_0} X_{v_j}$.  Since $X_{v_{k_0}}$ is the next element in the ordering, then  $X_{v_{k_0}}$ lies in a  simplex $\sigma$ that contains a simplex $\sigma'$ which is least in the Lusin--Sierpi\'nski ordering, and there is a vertex in $\sigma -\sigma'$, namely $v_{k_0}$, that is not in $\bigcup_{j < k_0} X_{v_j}$. 
 Let $X_{v_{j_1}}< \ldots < X_{v_{j_m}} < X_{v_{k_0}}$ be the boundaries of the 2-cells in $A_\beta$. 
We first deal with the case where all of the vertices $v_{j_1}, \ldots, v_{j_m}$ lie in $\sigma$. 
Then, by Lemma~\ref{lem:hellyimpliesniceunion}, $\bigcup_{v \in\sigma}C_v$ is simply connected, where the $C_v$ are 2-cells of $\mathcal{X}(P)$ with $\partial C_v=X_v$, so any closed path $p$ in the image of $A_\beta$ bounds a reduced disc diagram $E$ in $\bigcup_{v \in\sigma}C_v \subset \widetilde{\mathcal{X}(P)}$ (note that $\bigcup_{v \in\sigma}C_v$ itself need not be a disc diagram). 
We may choose $p$ so that  $\partial S=\partial E$, and consider $E \cup S$. Then either $S$ forms a cancellable pair with a 2-cell of $E$, contradicting the hypothesis of the claim, or the small-cancellation condition is violated. Indeed, $E$ is either a single cell\footnote{We warn the reader that, a priori, we cannot conclude from this observation that $S = E$. 
Indeed, if a relator $r$ in the presentation $P$ is a proper power, then $E$ and $S$ could be different lifts of the same disc --  the disc corresponding to $r$ -- in $\mathcal{X}(P)$.} or has at least 2 shells $S', S''$ by Greendlinger's Lemma (since any spurs can be removed), and in either case, the intersection $E \cap S$ or $S' \cap S$ is a single piece by Lemma~\ref{lem:c6cont}, thus contradicting the $C(6)$ condition.

 
 We may thus assume that there exists a cell $C$ of $A_\beta$ with $\partial C =  X_{v_{j_i}}$ for some $1 \leq i \leq m$ such that $v_{j_i}$ does not lie in $\sigma$. Note that,  by Lemma~\ref{clm:helly} and Lemma~\ref{clm:maxints}, there are at least four 2-cells in  $A_\beta$, so there is a 2-cell $C'$ in $A_\beta$ that is not adjacent to $X_{v_{k_0}}$, and by hypothesis, $\partial C' <X_{v_{k_0}}$. Now, either  $C'$  has the same property as $C$, namely $\partial C'$  corresponds to a vertex that does not lie in $\sigma$, or $C$ is also non-adjacent to  $X_{v_{k_0}}$, by the definition of the ordering. Thus, we may assume that $C=C'$.  Since $S$ is adjacent to $C$, then again by the definition of the ordering, $\partial S< X_{v_{k_0}}$  since the vertex corresponding to $\partial S$ forms a simplex  with $v_{j_i}$, and $v_{k_0}$ does not form a simplex with  $v_{j_i}$, and it cannot form a simplex with any lesser element in the ordering, since if it did, $X_{v_{k_0}}$ would not correspond to the maximal element in the collar.

  Thus, the base case of the induction is complete.
 
Assuming the claim when $\area (D_\beta) < N$, let $D_\beta$ be a disc diagram as hypothesised and with $\area (D_\beta) =N$. By Greendlinger's Lemma, $D_\beta$ is either a single cell, a ladder, or has at least $3$ shells or spurs. Since we may assume that $N>1$, and any spurs on $\partial D_\beta$ can be removed to reduce the number of  cells in $D_\beta$, we may assume that $D_\beta$  has at least $2$ shells. Let $S_1, \ldots, S_n$ be the shells of $D_\beta$. Consider the diagram $D_{\beta'}$ obtained from $D_\beta$ by removing  $S_1$, so $D_{\beta'}$ is collared by the annular diagram $A_{\beta'}$ obtained from  $A_\beta$ by removing the cells in the outerpath of $\partial S_1$ and adding $S_1$.
Every shell of $D_\beta - S_1$ is  a shell of  $D_{\beta'}$, so by the induction hypothesis,  $\partial S_i < X_{v_{M'}}$ for each shell $S_i$ with $1<i\leq n$, where $X_{v_{M'}}$ is the maximal element in the ordering corresponding to the boundary of a $2$-cell in $A_{\beta'}$. 

It may be, a priori, that $X_{v_{M'}}=\partial S_1$; we explain now why this isn't the case. Repeating the same construction as above to obtain a disc diagram $D_{\beta''}$  collared by an annular diagram $A_{\beta''}$, but this time removing a shell $S_j\neq S_1$, the induction hypothesis again implies that  $\partial S_i < X_{v_{M''}}$ for each shell $S_i$ with $1\leq i\leq n$ and $i\neq j$ and  where $X_{v_{M''}}$ is defined analogously to $X_{v_{M'}}$, i.e., it is the maximal element in the ordering corresponding to the boundary of a $2$-cell in $A_{\beta''}$.
If $X_{v_{M''}}=S_j$, then  $\partial S_j < \partial S_1$ and $\partial S_j > \partial S_1$, which is impossible. 
Thus, either $X_{v_{M'}}\neq \partial S_1$ or $X_{v_{M''}}\neq \partial S_j$, and since $X_{v_M}$ is the boundary of a $2$-cell of either $A_{\beta'}$ or $A_{\beta''}$ -- as it can only be excluded when removing one of the shells --, then $\partial S_i < X_{v_{k_0}}$ for each shell $S_i$ with $1 \leq i\leq n$, and the induction is complete.
\end{proof}

\begin{figure}
\centerline{\includegraphics[scale=0.67]{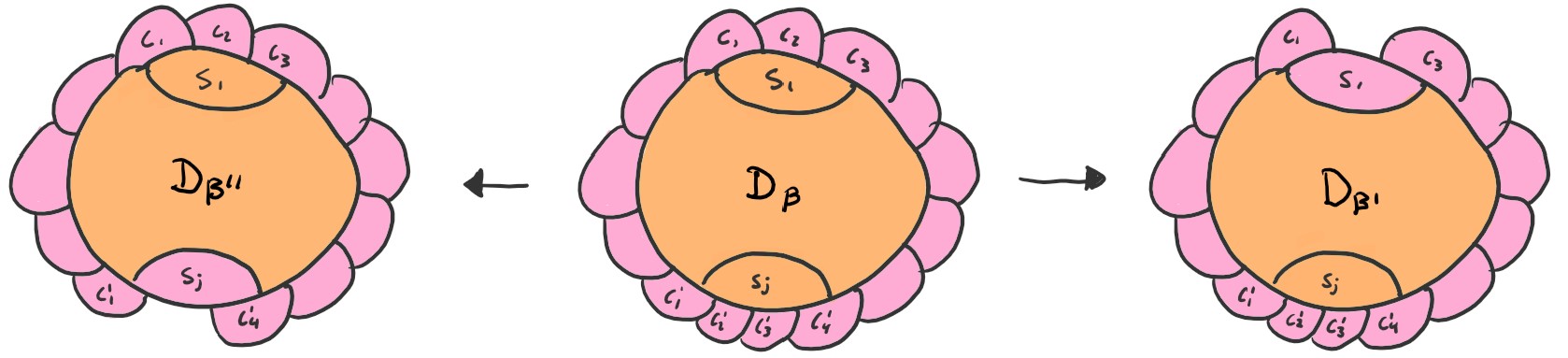}}
\caption{The reductions used in the final part of the proof of Lemma~\ref{clm:contractibleinduction} (replace $\beta$ with $\tau$ in the labelling) and in the inductive step of the proof of Claim~\ref{clm:final}.}
\label{fig:induction}
\end{figure}

We therefore conclude:

\begin{proof}[Proof of Theorem~\ref{thm:CLC6}]

By Lemma~\ref{lem:c6cont}, for each $k \in \naturals$, the intersection
$\bigcup_{j < k} X_{v_j} \cap X_{v_k}$ is contractible.
Abusing notation, let $e$ denote the basepoint of $\hat B_n$ (which is also the vertex representing the trivial element).
We prove, by induction on $k$, that this implies that there exist left coset representatives
$\{g_{v_1} \ldots, g_{v_k}\}$ for the left cosets of $N(\langle r_i\rangle)\nclose{r_1, \ldots, r_k}$ in $F(S)$  corresponding to the $X_{v_j}$ with $j \leq k$ such that   \begin{equation}\label{eq:cont proof}
    \pi_1 (\bigcup_{j \leq k} X_{v_j},e)=\langle \{r_i^{g_{v_j}}\}\rangle.
\end{equation}

For $k=1$, this is immediate since $X_{v_j}$ is a lift of some $r_i \in R$ based at $e$ of $\hat B_n$. We may thus take $g_{v_1}=e$ and the assertion follows.

Assume that~\eqref{eq:cont proof} holds for all $k \leq K$, and consider $\bigcup_{j \leq K} X_{v_j} \cup X_{v_{K+1}}$, where $X_{v_{K+1}}$ is a lift to $\hat B_n$ of the cycle  $c_{i_0}$ representing $r_{i_0} \in R$. 
Let $p $ be a point in the intersection  $\bigcup_{j \leq K} X_{v_j} \cap X_{v_{K+1}}$. Since this intersection is contractible,   then $\pi_1(\bigcup_{j \leq K} X_{v_j} \cup X_{v_{K+1}}, p)=\pi_1(\bigcup_{j \leq K} X_{v_j},p)\ast \pi_1(X_{v_{K+1}},p)$. Let $g\nclose{R}$ denote the left coset corresponding to  translating  $c_{i_0}$ to $X_{v_{K+1}}$. Since change of basepoint corresponds to conjugation, there is a $g_{v_{K+1}} \in g\nclose{R}$  such that

$$  \pi_1 (\bigcup_{j \leq K+1} X_{v_j},e)=\langle \{r_i^{g_{v_j}}\}_{i \in I, j \leq K}\cup\{r_{i_0}^{g_{v_{K+1}}}\}\rangle,$$

and the induction is complete.

Now since $\hat B_n$ is the infinite directed union 
$\hat B_n=\bigcup_{j \leq k,k \rightarrow \infty} X_{v_k}$, then $$\pi_1 (\hat B_n,e)=\lim_{k \rightarrow \infty}\pi_1 (\bigcup_{j \leq k} X_{v_k},e),$$ which finishes the proof of the theorem.
\end{proof}

\section{Appendix: the C(4)-T(4) and C(3)-T(6) cases
\\
\MakeLowercase{by} Macarena Arenas and Karol Duda}\label{appendix}

\begin{definition}\label{def:T}
 
Let $q$ be a natural number. We say that a $2$-complex $X$ satisfies the $T(q)$ \textit{small cancellation condition}  if there does not exist a reduced disc diagram $D \rightarrow X$ containing an internal vertex $v$ of valence $n$ where   $2<n<q$.
   
\end{definition}

We retain most of the notational conventions used in the main body of the paper.
One of the differences is in the definition of a shell.

\begin{definition}[$i$-shells]
An $i$-\emph{shell} of $D$ is a 2-cell $C \rightarrow D$
whose boundary path $\partial C \rightarrow D$ is a concatenation $qp_1 \cdots p_i$ for some $i \geq 1$
where $q$ is a boundary arc in $D$ and $p_1, \ldots , p_i$ are non-trivial pieces in the interior of $D$.
The arc $q$ is the \emph{outerpath} of $C$ and the concatenation $p_1 \cdots p_i$ is the \emph{innerpath} of $C$. 
\end{definition}

\subsection{The C(4)-T(4) case}

With that change both Greendlinger's Lemma and the Ladder Theorem hold also for $C(4)$--$T(4)$ complexes.

\begin{theorem}[Greendlinger's Lemma]\label{thm:greendlingerc4t4}
Let $X$ be a $C(4)$--$T(4)$ complex and $D \rightarrow X$ be a reduced disc diagram, then either
\begin{enumerate}
\item $D$ is a single cell,
\item $D$ is a ladder,
\item $D$ has at least three $1$-shells, $2$-shells and/or spurs.
\end{enumerate}
\end{theorem}

\begin{theorem}[The Ladder Theorem]\label{thm:ladderc4t4}
Let $X$ be a $C(4)$--$T(4)$ complex and $D \rightarrow X$ be a reduced disc diagram. If $D$ has exactly two $1$-shells or spurs, then $D$ is a ladder.
\end{theorem}

For the proof of the analogue of Lemma \ref{clm:contractibleinduction} we need a slightly more refined version of Greendlinger's Lemma, giving us additional information about the shells.

\begin{theorem}\label{thm:corners}
Let $X$ be a $C(4)$--$T(4)$ complex and $D \rightarrow X$ be a reduced disc diagram. If $D$ does not contain spurs, then it is either a single cell, or contains a $1$-shell, or contains a $2$-shell in $D$ containing  a boundary vertex of $D$ of degree exactly $3$.
\end{theorem}

\begin{myproof}
    Assume that $\area(D)\geq 2$.
    Consider a complex $D'$ obtained from $D$ by removing all vertices of degree $2$, with the exception of two on the outerpath of every $1$-shell and a single one on the outerpath of every $2$-shell.
    Since $D'$ is a disc diagram, the Euler characteristic $\chi(D')$ is equal to $1$.
    We distribute the Euler characteristic from each edge as $\frac{-1}{2}$ to both of it ends, and from each cell $\frac{1}{p}$ to each of its $p$ boundary vertices.
    It is clear that the sum of the distributed Euler characteristic over all vertices is $1$.
    Let $v$ be an internal vertex of $D'$.
    Since $D$ is reduced and satisfies $T(4)$, the degree of $v$ is $n\geq 4$. The number of cells containing $v$ is also $n$. 
    Thus, the distributed Euler characteristic for $v$ is at most $1-\frac{n}{2}+\frac{n}{4}$ which for $n\geq 4$ is non-positive.
    Let $u$ be a boundary vertex of $D'$. 
    The degree of $u$ is $n\geq 2$ and it is greater by at least $1$ than number of cells containing $u$.
    Thus, the distributed Euler characteristic for $u$ is at most $1-\frac{n}{2}+\frac{n-1}{4}$, where $n$ is degree of $u$.
    The only positive value it can have is $\frac{1}{4}$ which happens exactly when $n=2$ and $u$ is contained in a cell. 
    This happens for boundary vertices of degree two in $1$-shells and $2$-shells.
    Since the sum is $1$ and all other values are non-positive, there are multiple cells that are either a $1$-shell or a $2$-shell in $D$.
    Moreover observe that if all of these cells are $2$-shells, then some of them have a boundary vertex of degree exactly $3$. 
    Indeed, otherwise each of these $2$-shells has two boundary vertices of degree at least $4$. The distributed Euler characteristic of these vertices is at most $1-\frac{n}{2}+\frac{n-1}{4}$ which is at most $\frac{-1}{4}$. Since each such vertex can be shared by at most two $2$-shells, the sum of the distributed Euler characteristic is at most $0$, a contradiction.  
\end{myproof}

Note that Lemma \ref{lem:c6cont} is true for the $C(4)$--$T(4)$ condition as well.

\begin{lemma}\label{lem:c4t4cont}
Let $X$ be a simply-connected $C(4)$--$T(4)$ small-cancellation complex. Let $C_1,C_2$ be 2-cells of $X$. Then  either $\partial C_1 =\partial C_2$, or $ C_1\cap  C_2=\emptyset$, or $ C_1\cap  C_2$ is contractible.
\end{lemma}

Analogues of Lemmas \ref{clm:helly} and $\ref{clm:maxints}$ exist for $C(4)$--$T(4)$ complexes. The Helly property for triple intersections can be proved using induction and Greendlinger's Lemma, and can also be found in \cite[Proposition 3.7]{hoda2019quadric}. Note that this result combined with Lemma \ref{lem:c4t4cont} is enough to run the proof in~\cite{OP18}. While this proof is phrased for the $C(6)$ case, it only uses the fact that any counterexample to the Helly Property would be a disc diagram with boundary consisting of at most $3$ pieces, which contradicts the $C(4)$ condition as well. We thus obtain:

\begin{lemma}\label{clm:hellyc4t4}
Let $\mathcal{X}(P)$ be the presentation complex associated to a $C(4)$--$T(4)$ small-cancellation presentation $P$, let $\Lambda$ be its structure graph, and let $V' \subset V(\Lambda)$ with $|V'|<\infty$. If each pairwise intersection $ X_v \cap X_v'$ with $v,v' \in V'$ is non-empty, then the total intersection $\bigcap_{v \in V'}X_v$ is non-empty. 
\end{lemma}

We can improve the conclusion of Lemma~\ref{clm:hellyc4t4} to deduce contractibility of the total intersection, as in the $C(6)$ case. A version of this result follows from putting together Propositions 3.5 to 3.8 in~\cite{hoda2019quadric}.

\begin{lemma}\label{clm:maxintsc4t4}
Let $\mathcal{X}(P)$ be the presentation complex associated to a $C(4)$--$T(4)$ small-cancellation presentation $P$ and let $\Lambda$ be its structure graph. For all $I \subset V(\Lambda)$ with $\infty >|I|\geq 2$, the intersection $\bigcap_{v \in I} X_v $ is non-empty if and only if it is connected (and hence contractible). 
\end{lemma}

The proof of  Lemma~\ref{clm:maxintsc4t4} follows the proof of Lemma~\ref{clm:maxints}. The difference is in the proof of the fact that $D^+$ is a degenerate disc diagram.

Like in the previous proof, note that the choices of $\alpha$ and $\beta$ imply that $\partial D$ cannot contain any spurs, as such spurs could be removed to shorten $\alpha$ and/or $\beta$.
Thus by Greendlinger's Lemma $D$ must contain either a $1$-shell $S_1$, or at least three $2$-shells $S_2,S_3,S_4$. The outerpath of either of them  is either a subpath of $\alpha$, a subpath of $\beta$, or contains either $x$ or $y$ as an internal vertex of the path. 
In the first and second cases, $\partial S_i \cap \partial D$ is a single piece, and the innerpath of $S_i$ is the concatenation of at most $2$ pieces, so $\partial S$ is the concatenation of at most $3$ pieces, contradicting the $C(4)$ condition. 
In the third case, $\partial S_i \cap \partial D$ is the concatenation of at most $2$ pieces. In the case of a $1$-shell $S_1$ the innerpath consist of exactly a single piece, so $\partial S_1$ is the concatenation of at most $3$ pieces, contradicting the $C(4)$ condition. A $2$-shell containing either $x$ or $y$ can exist without contradicting $C(4)$, but then there are at least three $2$-shells $S_1, S_2, S_3$ in $D$, while only two of them can contain either $x$ or $y$ as an internal vertex. Thus there is at least one $2$-shell with an outerpath being either a subpath of $\alpha$, a subpath of $\beta$, yielding again a contradiction.

Lemma~\ref{lem:hellyimpliesniceunion} also holds in this case:

\begin{lemma}\label{lem:hellyimpliesniceunion4}
Let $\mathcal{X}(P)$ be the presentation complex associated to a $C(4)-T(4)$ presentation $P$ and let $\Lambda$ be its structure graph. Let $\sigma$ be a simplex of $\mathbf{N}(\mathcal{U})$ and let $Y_\sigma$  be the subcomplex of $\mathcal{X}(P)$ consisting of all the 2-cells whose boundary corresponds  to vertices in $\sigma$. Then $Y_\sigma$ is simply connected.
\end{lemma}

We use the results collected above to deduce an analogue of Claim~\ref{clm:final}.

\begin{claim}\label{clm:finalc4t4} 
 Let $\mathcal{X}(P)$ be the presentation complex associated to a $C(4)$--$T(4)$ small-cancellation presentation $P$, let $V=\{v_j\}_{j \in \naturals}$ be the vertices of its structure graph $\Lambda$, and for each $v_j \in V$, let $X_{v_j}$ be the corresponding element of $\mathcal{U}$.  Let $A_\beta$ be an annular diagram in $\widetilde{\mathcal{X}(P)}$ collaring a reduced disc diagram $D_\beta$ so that $\partial D_\beta= D_\beta\cap A_\beta=\beta$, and finally, let $X_{v_{k_0}}$ be the maximal element in the ordering of Definition~\ref{def:order} corresponding to the boundary of a $2$-cell in $A_\beta$. Then $\partial S < X_{v_{k_0}}$  for every shell $S$ in $D_\beta$.
\end{claim}

\begin{myproof}
The proof is analogous to that of Claim~\ref{clm:final}, except that we substitute Lemma~\ref{lem:c6cont}, Lemma~\ref{clm:helly},  Lemma~\ref{clm:maxints}, and Lemma~\ref{lem:hellyimpliesniceunion} with their $C(4)-T(4)$ analogues -- Lemma~\ref{lem:c4t4cont},   Lemma~\ref{clm:hellyc4t4},  Lemma~\ref{clm:maxintsc4t4},and Lemma~\ref{lem:hellyimpliesniceunion4}. The key small-cancellation point of the argument  is that in a reduced disc diagram $D \rightarrow \mathcal{X}(P)$, the outerpath of a shell cannot be a single piece, and this is also true for $1$-shells and $2$-shells in the present case, as well  as in the case of $D$ being one  cell.
\end{myproof}

Finally, the $C(4)-T(4)$ version of Theorem~\ref{thm:CLC6intro} follows verbatim from the proof of Theorem~\ref{thm:CLC6} once we have the analogue of Lemma \ref{clm:contractibleinduction}.

\begin{lemma}\label{clm:contractibleinductionc4t4} Let $\mathcal{X}(P)$ be the presentation complex associated to a $C(4)$--$T(4)$ small-cancellation presentation $P$, let $V=\{v_j\}_{j \in \naturals}$ be the vertices of its structure graph $\Lambda$, and for each $v_j \in V$, let $X_{v_j}$ be the corresponding element of $\mathcal{U}$. For each $k \in \naturals$, the intersection
$$\bigcup_{j < k} X_{v_j} \cap X_{v_k}$$ is contractible.
\end{lemma}

\begin{myproof}

The proof is by induction on $k$, and it follows closely the proof of  Lemma \ref{clm:contractibleinduction}. As in that proof, for the inductive step we assume that $\bigcup_{j < k} X_{v_j} \cap X_{v_k}$ is contractible for  all $k < k_0$, and we need to show that $\bigcup_{j < k_0} X_{v_j} \cap X_{v_{k_0}}$ is contractible. We assume towards a contradiction that it is not, so $\bigcup_{j < k_0} X_{v_j} \cap X_{v_{k_0}}$ is either disconnected, or connected but not simply-connected.

 The non-simply-connected case is precluded by the induction hypothesis by the Seifert Van-Kampen Theorem (note that the argument given in Lemma~\ref{clm:contractibleinduction} does not use small-cancellation in any explicit way).

Now if  the intersection $\bigcup_{j < k_0} X_{v_j} \cap X_{v_{k_0}}$  is disconnected, again we follow the strategy of the proof in the $C(6)$ case and, retaining the notation in that proof, let $D_\tau$ denote a reduced disc diagram collared by a chain of cells $\bigcup_{j < k_0} X_{v_j}$ and by $X_{v_{k_0}}$. 
The contradiction is now reached by considering the area of $D_\tau$. The case of $\area (D_\tau)= 0$ is straightforward and identical to the $C(6)$ case. The differences are in the case where $\area (D_\tau)\geq 1$ (specifically, in the second to last paragraph in the proof Lemma~\ref{clm:contractibleinduction}). We explain this part of the argument in detail.

If $\area (D_\tau)\geq 1$ we may assume, by performing the same reductions to $\tau$ as in the $\area (D_\tau)= 0$ case, that $\partial D_\tau$ has no spurs, so by Theorem \ref{thm:corners}  $D_\tau$ contains a cell $S$ which either is the only cell of $D_\tau$, or is a $1$-shell or a $2$-shell containing a vertex of degree $3$ in $\partial S\cap \partial D$. Claim~\ref{clm:finalc4t4} shows that $\partial S < X_{v_{k_0}}$ in $D_\tau$. To finish the proof of Lemma~\ref{clm:contractibleinduction},
observe that if $S$ is the only cell of $D_\tau$ then the $C(4)$ condition implies that $S$ intersects at least $3$ consecutive cells $C_1,C_2,C_3,C_4$ of $A_\tau$. If $S$ is a $1$-shell, then by the $C(4)$ condition it intersects at least $3$ consecutive cells $C_1,C_2,C_3$ of $A_\tau$.
If $S$ is a $2$-shell, then by the  $C(4)$--$T(4)$ condition it intersects at least $3$ consecutive cells $C_1,C_2,C_3$ of $A_\tau$.
In any case, the path $\tau'$ obtained from $\tau$ by pushing across $C_2$ traverses at most as many cells as $\tau$, but bounds a disc diagram $D_{\tau'}$ with $\area (D_{\tau'})< \area (D_\tau)$, contradicting our initial choices. 

Thus we arrive at a contradiction in all cases, so  $\bigcup_{j < k_0} X_{v_j} \cap X_{v_{k_0}}$ must be contractible and the induction is complete. \end{myproof}

\subsection{The C(3)-T(6) case}

The statement of Greendlinger's Lemma in the case of $C(3)$--$T(6)$ complexes is more complicated.

\begin{theorem}[Greendlinger's Lemma]\label{thm:greendlingerc3t6}
Let $X$ be a $C(3)$--$T(6)$ complex and $D \rightarrow X$ be a reduced disc diagram, then either
\begin{enumerate}
\item $D$ is a single cell,
\item $D$ is a ladder,
\item $D$ has at least three spurs, $1$-shells and/or pairs of $2$-shells sharing an edge incident to a boundary vertex.
\end{enumerate}
\end{theorem}

Note that Lemma \ref{lem:c6cont} is true for the $C(3)$--$T(6)$ condition as well.

\begin{lemma}\label{lem:c3t6cont}
Let $X$ be a simply-connected $C(3)$--$T(6)$ small-cancellation complex. Let $C_1,C_2$ be 2-cells of $X$. Then  either $\partial C_1 =\partial C_2$, or $ C_1\cap  C_2=\emptyset$, or $ C_1\cap  C_2$ is contractible.
\end{lemma}

Moreover, in the case of $C(3)$--$T(6)$ complexes, we have some additional information about pieces of a complex, a fact which was first observed by Pride (see also \cite{GerstenShort90}, \cite{McCammondWiseFanLadder}).

\begin{lemma}\label{t5piecesedges}
If $X$ is a $T(5)$ complex then every piece in $X$ is an edge.
\end{lemma}

Consequently, the statement of Lemma \ref{lem:c3t6cont} can be refined to the following: either $\partial C_1 =\partial C_2$, or $ C_1\cap  C_2=\emptyset$, or $ C_1\cap  C_2$ is  a single edge or a single vertex.

As in $C(4)$--$T(4)$ case, for the analogue of the Lemma \ref{clm:contractibleinduction}, we need a refined version of Greendlinger's Lemma.

\begin{theorem}\label{thm:cornersc3t6}
Let $X$ be a $C(3)$--$T(6)$ complex and $D \rightarrow X$ be a reduced disc diagram. If $D$ does not contain spurs then it  is either a single cell, contains a $1$-shell with at most one vertex of degree more than $6$ in $D$, or contains  a $2$-shell with all boundary vertices of degree at most $5$ in $D$.
\end{theorem}

\begin{myprooft}
    The proof is very similar to the proof of Theorem \ref{thm:corners}. Let $D \rightarrow X$ be a reduced disc diagram with $area(D)\geq 2$. Since the $T(6)$ condition implies that all pieces have length 1, there are no internal vertices of degree $2$ in $D$. To simplify the calculations, let $D'$ be the disc diagram obtained from $D$ by removing all vertices of degree $2$ at the boundary of $D$, with the exception of a single vertex on the outerpath of every $1$-shell.
    Since $D'$ is a disc diagram, the Euler characteristic $\chi(D')$ is equal to $1$.
    We distribute the Euler characteristic from each edge as $\frac{-1}{2}$ to both of it ends, and from each cell $\frac{1}{p}$ to each of its $p$ boundary vertices.
    It is clear that the sum of the distributed Euler characteristic over all vertices is  equal to $1$.
    Let $v$ be an internal vertex of $D'$.
    Since $D$ is reduced and satisfies the $T(6)$ condition, the degree of $v$ is $n\geq 6$. The number of cells containing $v$ is also $n$. 
    Thus, the distributed Euler characteristic for $v$ is at most $1-\frac{n}{2}+\frac{n}{3}$ which for $n\geq 6$ is non-positive.
    Let $u$ be a boundary vertex of $D'$. 
    The degree of $u$ is $n\geq 2$ and it is greater by at least $1$ than the  number of cells containing $u$.
    Thus, the distributed Euler characteristic for $u$ is at most $1-\frac{n}{2}+\frac{n-1}{3}$, where $n$ is the degree of $u$.
    Consequently, the  distributed Euler characteristic for $u$ has is non-positive  if $n\geq 4$. The possible positive values are:
    \begin{itemize}
        \item $\frac{1}{3}$ for $n=2$ which can happen only for a boundary vertex of degree two in a $1$-shell;
        \item $\frac{1}{6}$ for $n=3$ if $u$ is a vertex at the end of the piece lying in  two triangles, and these triangles are either $1$-shells or $2$-shells;
        \item $\frac{1}{12}$ for $n=3$ if $u$ is a vertex at the end of the piece shared by a triangle and a square, such triangle is either a $1$-shell or a $2$-shell, a square is either a $3$-shell or the intersection of its boundary with the  boundary of $D'$ has two components.
     \end{itemize}
    Since the sum is $1$ and the distributed Euler characteristic of all other vertices is non-positive, there are multiple cells in $D$ that are either $1$-shells or $2$-shells.
    If a vertex has degree at least $6$, then its distributed Euler characteristics at most $1-\frac{n}{2}+\frac{n-1}{3}$ which cannot exceed $\frac{-1}{3}$.
    Vertices of degree at least $6$ can be shared by two $1$-shells, or by two $2$-shells or a $1$-shell and a $2$-shell. In either case if none of the $1$-shells touch a vertex of degree less than $6$ on at least one side, and none of $2$-shells touch vertices of degree less than $6$ on both sides, then the sum of the distributed Euler characteristic of vertices is at most $0$, which is a contradiction.
\end{myprooft}

In the $C(3)$--$T(6)$ case, the adaptations to the proof of Theorem~\ref{thm:CLC6} are a bit less straightforward -- the main difficulty arises from the fact that the Helly Property does not hold in full generality, as illustrated in Figure~\ref{fig:triforce}. To bypass this problem, we prove a weaker form of the Helly Property and  we refine the ordering of the nerve complex of $\mathcal{U}$. 

\begin{figure}
\centerline{\includegraphics[scale=.7]{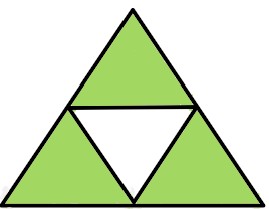}}
\caption{An example of a $C(3)-T(6)$  for which the usual Helly Property does not hold: the coloured triangles pairwise intersect at vertices, but their total intersection is empty.}
\label{fig:triforce}
\end{figure}

\begin{lemma}\label{clm:hellytriangular}
Let $\mathcal{X}(P)$ be the presentation complex associated to a $C(3)-T(6)$ small-cancellation presentation $P$, let $\Lambda$ be its structure graph, and let $V' \subset V(\Lambda)$ with $|V'|< \infty$. If each pairwise intersection $ X_v \cap X_{v'}$ with $v,v' \in V'$ is an edge, then the total intersection $\bigcap_{v \in V'}X_v$ is either a single edge or a single vertex. 
\end{lemma}

\begin{myprooft} 
    The proof is by induction on $|V'|$ where the base case is $|V'|=3$.
    Let $V'=\{X_{v_1}, X_{v_2}, X_{v_3}\}$ and let $e_{ij} \subset X_{v_i} \cap X_{v_j}$ be the edges in each pairwise intersection.
    Observe that if $e_{ij}\cap e_{ik}\neq \emptyset$, then by Lemma~\ref{lem:c3t6cont}, the triple intersection is also not empty and is contractible, which means that it is either a single edge or a single vertex.
    Thus we can assume that $e_{ij}\cap e_{ik}=\emptyset$ for any permutation of $i,j,k$.
    
    For each $i$, let $\rho_i \rightarrow X_{v_i}$ be a spurless path joining $e_{ij}, e_{ik}$.
    Moreover, since $e_{ij}\cap e_{ik}=\emptyset$ we can choose the $\rho_i$ so that the concatenation $\rho_1\rho_2\rho_3$ is a cycle that does not separate any two of the $e_{ij}$. Let $D\rightarrow \hat B_n$ be a reduced disc diagram with $\partial D=\rho_1\rho_2\rho_3$.
    Then by Lemma \ref{thm:greendlingerc3t6}, $D$ is a single cell, or is a ladder (and contains exactly two $1$-shells), or it contains at least three $1$-shells and/or pairs of $2$-shells sharing an edge incident to a boundary vertex.

    It is clear that neither of $X_{v_1}, X_{v_2}, X_{v_3}$ can be a $1$- or $2$-shell in $D$, as 
    $e_{ij}\cap e_{ik}=\emptyset$ implies at least three pieces in the part of its boundary in the interior of $D$. 
    If $D$ is a single cell $C$ or contains a $1$-shell $C$, then there is a reduced disc diagram  $C\cup X_{v_i} \cup X_{v_j}$ with an internal vertex of degree $3$, contradicting the $T(6)$ condition.
    If $D$ contains a pair of $2$-shells $C_1,C_2$ sharing an edge incident to a boundary vertex, then there is either a reduced disc diagram  $C_1\cup C_2\cup X_{v_i}$ with an internal vertex of degree $3$ or $C_1\cup C_2\cup X_{v_i} \cup X_{v_j}$ with an internal vertex of degree $4$, in either case contradicting $T(6)$.

    Thus  $D$ must be a degenerate disc diagram, but then $e_{ij}\cap e_{ik}\neq \emptyset$.

    For the inductive step, we use the fact that $\bigcap_{v \in V'-\{v_1\}}X_v$ 
    is either a single edge or vertex, and now intersect $X_{v_1}$ with $\bigcap_{v \in V'-\{v_1,v_2\}}X_v$.
\end{myprooft}

Let $\hat B_n^o=\hat B_n-\hat B_n^{(0)}$, let $\mathcal{U}^o$ be the open cover of $\hat B_n^o$ induced by $\mathcal{U}$,  let $\mathcal{V}^o$ be the vertex set of the nerve complex $\mathbf{N}(\mathcal{U}^o)$. In general, $\hat B_n^o$ need not be connected, and the connected components of $\hat B_n^o$ partition $\mathcal{V}^o$ into subsets, which we denote $V^o_\nu$. Let $\varphi^o_\nu:V^o_\nu \rightarrow \naturals$ denote the ordering of Definition~\ref{def:order} on $\mathbf{N}(\mathcal{U}^o_\nu)$.  
The vertex set $V^o_\nu$ of $\mathbf{N}(\mathcal{U}^o_\nu)$  naturally corresponds to a subset (which we also denote $V^o_\nu$) of the vertex set $V$ of $\mathbf{N}(\mathcal{U})$.  We make the following observation:

\begin{lemma}\label{lem:unth}
    If $X_v, X_w \in \mathcal{U}$ intersect only at  vertices, then $X_v \cap X_w$ is a single vertex. 
\end{lemma}

\begin{myprooft}
Suppose $\{p,q\} \subset X_v\cap X_w$, and let $\sigma_v \rightarrow X_v, \sigma_w\rightarrow X_w$ be spurless paths connecting $p,q$.
Since $\mathcal{X}(P)$ is simply-connected, there is a   disc diagram   $D \rightarrow X$ with $\partial D=\sigma_v\sigma_w^{-1}$. If $area(D)=0$, then there is an edge $e \subset  X_v\cap X_w$, contradicting our hypothesis. Suppose $area(D)>0$. Then either $X_v$ or $X_w$ is an untethered component by Lemma~\ref{lem:c3t6cont}. But then a 2-cell of $D$ with outerpath on either $\sigma_v \rightarrow X_v$ or $\sigma_w\rightarrow X_w$ contradicts the definition of an untethered component.  
    We have thus established that $ X_v\cap X_w$ is a single vertex.
    \end{myprooft}

We make the following observation:

\begin{remark}\label{rmk:orders comp}
If two elements   $X_v, X_w \in \mathcal{U}$ intersect at a vertex $p$, then since $\hat B_n$ is a Cayley graph, then either $p$ is a cut vertex of $\hat B_n$, or $X_v$ and $X_w$ lie in the same connected component of $\hat B_n^o$. Indeed, when viewed in $\mathcal{X}(P)$, the projections of $X_v$ and $X_w$ (which we denote in the same way) are either 2-cells or loops in $\mathcal{X}(P)^{(1)}$, and either $\mathcal{X}(P)$ is a wedge with  $X_v$ and $X_w$ in different components of the wedge, or  $X_v$ and $X_w$  are both 2-cells, and there is a non-trivial relation between a generator appearing in the attaching map of $X_v$ and a generator appearing in the attaching map of $X_w$. 
\end{remark}

\begin{figure}
\centerline{\includegraphics[scale=0.45]{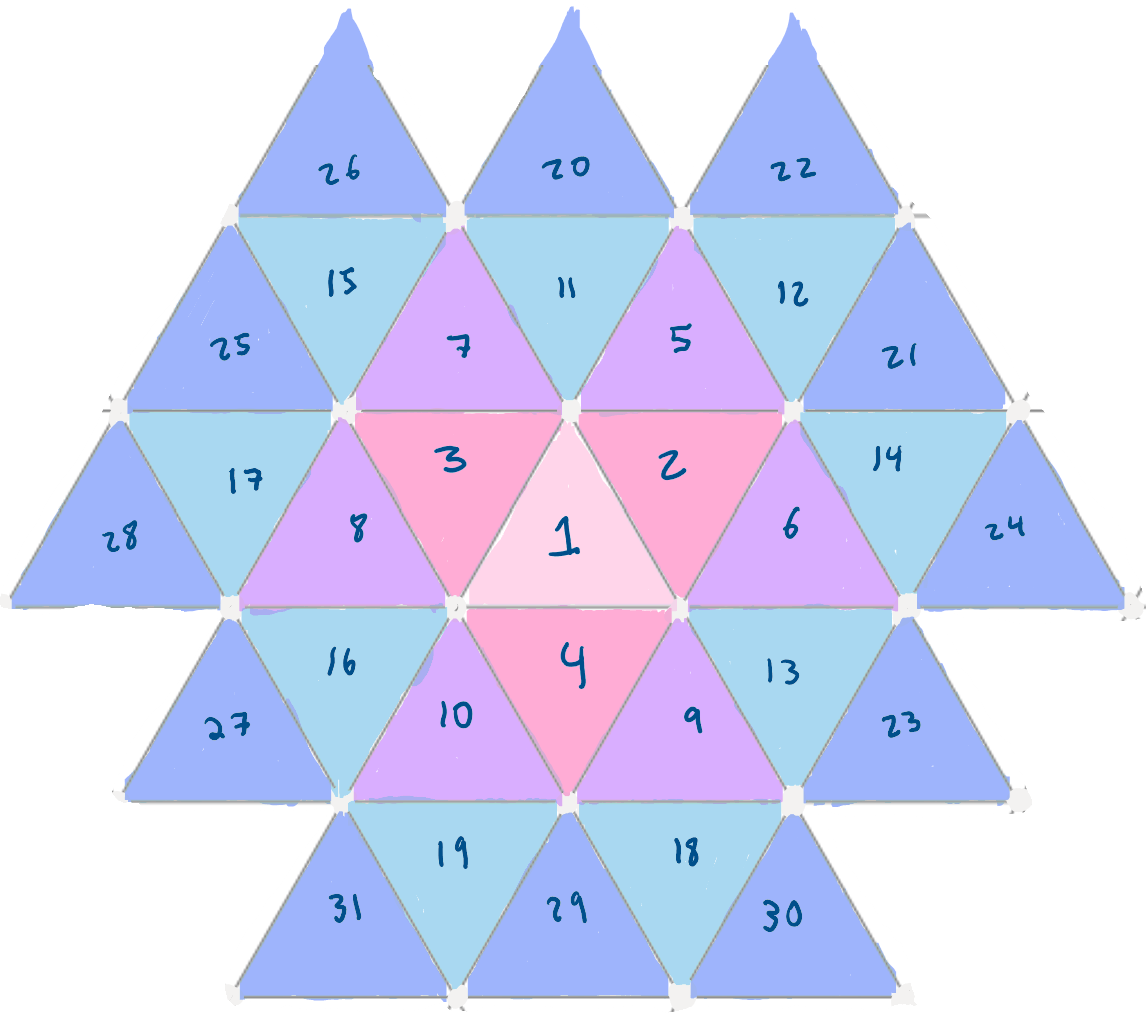}}
\caption{The modified ordering for a portion of a triangular tiling.}
\label{fig:orderingexc3t6}
\end{figure}

The previous remark shows that the  orders  $\{\varphi^o_\nu\}$  defined on the $V^o_\nu$ are compatible, and induce a  partial  order on $\mathcal{U}$ that extends to a total order $\varphi:V\rightarrow \naturals$ by the order-extension principle.

We re-index the vertex sets $V^o_\nu$ according to this order, i.e., we write $V^o_\nu=\{v^o_{\nu,j}\}_{j \in \naturals}$ (or simply $V^o_\nu=\{v^o_{\nu,j}\}_{j \in \naturals}$ when the $V^o_\nu$ is fixed) where the indexing agrees with $\varphi^o_\nu$. As in previous sections, we refer to this order and the induced  order on the corresponding subset of $\mathcal{U}$ 
interchangeably.

We  prove Theorem~\ref{thm:CLC6intro} for $C(3)-T(6)$ presentations by first showing that, for each $V^o_\nu\subset V$, the intersection $$\bigcup_{j < k} X_{v_j} \cap  X_{v_k}$$ 
is contractible for each $k \in \naturals$. We then assemble the various subcomplexes corresponding to the $V^o_\nu\subset V$ together using the ordering $\varphi$ obtained by the order-extension principle. This  yields the desired homotopy equivalence by Remark~\ref{rmk:orders comp}.

In the proof, as in the $C(6)$ and $C(4)-T(4)$ cases, we use the following Lemma, which is proven in an identical way to Lemma~\ref{lem:hellyimpliesniceunion}.

\begin{lemma}\label{lem:hellyimpliesniceunion3}
Let $\mathcal{X}(P)$ be the presentation complex associated to a $C(3)-T(6)$ presentation $P$ and let $\Lambda$ be its structure graph. Let $\sigma$ be a simplex of $\mathbf{N}(\mathcal{U})$ and let $Y_\sigma$  be the subcomplex of $\mathcal{X}(P)$ consisting of all the 2-cells whose boundary corresponds  to vertices in $\sigma$. Then $Y_\sigma$ is simply connected.
\end{lemma}

We now apply the results collected above to deduce an analogue of Claim~\ref{clm:final}.

\begin{claim}\label{clm:finalc3t6} 
 Let $\mathcal{X}(P)$ be the presentation complex associated to a $C(3)$--$T(6)$ small-cancellation presentation $P$. Fix a $\nu$ and consider  $V^o_\nu=\{v_j\}_{j \in \naturals}$.  For each $v_j \in V^o_\nu$, let $X_{v_j}$ be the corresponding element of $\mathcal{U}$.  Let $A_\beta$ be an annular diagram in $\widetilde{\mathcal{X}(P)}$ collaring a reduced disc diagram $D_\beta$ so that $\partial D_\beta= D_\beta\cap A_\beta=\beta$, and finally, let $X_{v_{k_0}}$ be the maximal element in the ordering $\varphi^o_\nu$ corresponding to the boundary of a $2$-cell in $A_\beta$. Then $\partial S < X_{v_{k_0}}$  for every shell $S$ in $D_\beta$.
\end{claim}

\begin{myprooft}
The proof is analogous to that of Claim~\ref{clm:final}, except that we substitute Lemma~\ref{lem:c6cont}, Lemma~\ref{clm:helly},  Lemma~\ref{clm:maxints}, and Lemma~\ref{lem:hellyimpliesniceunion} with their $C(3)-T(6)$ analogues -- Lemma~\ref{lem:c3t6cont}, Lemma~\ref{clm:hellytriangular},  Lemma~\ref{lem:unth}, and Lemma~\ref{lem:hellyimpliesniceunion3}. There is a difference in a key small-cancellation point of the argument: the outerpath of a shell can be a single piece in the case of a $2$-shell. But by Theorem~\ref{thm:cornersc3t6}, if we do not have at least two $1$-shells, then there are at least two pairs of $2$-shells sharing an edge incident to a boundary vertex. Such pair of $2$-shells cannot be bounded by $\partial S$, as this would contradict the $T(6)$ condition.
\end{myprooft}

We can now conclude:

\begin{lemma}\label{clm:contractibleinduction3} Let $\mathcal{X}(P)$ be the presentation complex associated to a $C(3)-T(6)$ small-cancellation presentation $P$. Fix a $\nu$ and let $V^o_\nu=\{v_j\}_{j \in \naturals}$. For each $v_j \in V^o_\nu$, let $X_{v_j}$ be the corresponding element of $\mathcal{U}$. For each $k \in \naturals$, the intersection
$$\bigcup_{j < k} X_{v_j} \cap X_{v_k}$$ is contractible.
\end{lemma}

\begin{myprooft}
    This proof follows the scheme of the proof in the $C(4)$--$T(4)$ case. Again we prove that we have a cell $S$ that intersects at least $3$ consecutive cells $C_1,C_2,C_3$ of $A_\tau$. 
    By Theorem \ref{thm:cornersc3t6}, there is a cell $S$ in $D_\tau$ that is either the only cell of a $D_\tau$, or a 1-shell touching a vertex of degree at most $5$, or a $2$-shell touching two boundary vertices of degree at most $6$. If $S$ is the only cell, the fact that $S$ intersects  at least three consecutive shells follows from the  $C(3)$ condition; in the case of a $1$-shell, intersection with two  of the consecutive shells  follows from the $C(3)$ condition, and intersection with the third one from the $T(6)$ condition; in the case of a $2$-shell, it follows from the $T(6)$ condition and the fact that such a shell has at least one boundary edge.
\end{myprooft}

By Remark~\ref{rmk:orders comp}, distinct connected components of $\hat B_n^o$ intersect in a single vertex; we immediately deduce an analogue to Lemmas~\ref{clm:contractibleinduction} and~\ref{clm:contractibleinductionc4t4}: 

\begin{corollary}\label{cor:contractibleinduction3} Let $\mathcal{X}(P)$ be the presentation complex associated to a $C(3)-T(6)$ small-cancellation presentation $P$. Let $V=\{v_j\}_{j \in \naturals}$ be the vertices of its structure graph $\Lambda$. For each $v_j \in V$, let $X_{v_j}$ be the corresponding element of $\mathcal{U}$. For each $k \in \naturals$, the intersection
$$\bigcup_{j < k} X_{v_j} \cap X_{v_k}$$ is contractible.
\end{corollary}

Reasoning in exactly the same way as in the proof of Theorem~\ref{thm:CLC6}, the $C(3)-T(6)$ version of Theorem~\ref{thm:CLC6intro} now follows.

\bibliographystyle{alpha}
\bibliography{bib9.bib}
\end{document}